\documentclass[10pt,reqno]{amsart}

\usepackage{amsmath,amssymb,amsthm,graphicx,mathrsfs}
\usepackage{url}
\usepackage{xcolor}

\usepackage[T1]{fontenc}
\usepackage[utf8]{inputenc}
\usepackage[english]{babel}

\usepackage[margin=1.5in]{geometry}

\usepackage{booktabs}

\usepackage[section]{placeins}

\usepackage[colorlinks=true,linkcolor=black,citecolor=black,urlcolor=magenta]{hyperref}
\newcommand\colorurl[2]{{\color{#1}\url{#2}}}

\usepackage{pgfplots}

\usepackage[ruled]{algorithm2e}

\newtheorem*{thm}{Theorem}

\theoremstyle{definition}

\newtheorem*{rmk}{Remark}

\newcommand{\N}{\mathbb N}

\makeatletter
\def\@setauthors{%
  \begingroup
  \def\thanks{\protect\thanks@warning}%
  \trivlist
  \centering\large \@topsep30\p@\relax
  \advance\@topsep by -\baselineskip
  \item\relax
  \author@andify\authors
  \def\\{\protect\linebreak}%
  \authors%
  \ifx\@empty\contribs
  \else
    ,\penalty-3 \space \@setcontribs
    \@closetoccontribs
  \fi
  \endtrivlist
  \endgroup
}
\def\@settitle{\begin{center}%
  \baselineskip14\p@\relax
    \normalfont\LARGE
  \@title
  \end{center}%
}
\makeatother

\pgfplotsset{compat=1.18}

\begin{document}

\date{\today}

\title[Warehouse storage and retrieval optimization]{Warehouse storage and retrieval optimization via clustering, dynamical systems modeling, and GPU-accelerated routing}

\author{Magnus Bengtsson}
\address[Magnus Bengtsson]{Department of Engineering, University of Bor{\aa}s, SE-501 90 Bor{\aa}s, Sweden}
\email{magnus.bengtsson@hb.se}

\author{Jens Wittsten}
\address[Jens Wittsten]{Department of Engineering, University of Bor{\aa}s, SE-501 90 Bor{\aa}s, Sweden}
\email{jens.wittsten@hb.se}

\author{Jonas Waidringer}
\address[Jonas Waidringer]{Department of Engineering, University of Bor{\aa}s, SE-501 90 Bor{\aa}s, Sweden}
\email{jonas.waidringer@hb.se}

\begin{abstract}
This paper introduces a warehouse optimization procedure aimed at enhancing the efficiency of product storage and retrieval. By representing product locations and order flows within a time-evolving graph structure, we employ unsupervised clustering to define and refine compact order regions, effectively reducing picking distances. We describe the procedure using a dynamic mathematical model formulated using tools from random dynamical systems theory, enabling a principled analysis of the system's behavior over time even under random operational variations. For routing within this framework, we implement a parallelized Bellman-Ford algorithm, utilizing GPU acceleration to evaluate path segments efficiently. To address scalability challenges inherent in large routing graphs, we introduce a segmentation strategy that preserves performance while maintaining tractable memory requirements. Our results demonstrate significant improvements in both operational efficiency and computational feasibility for large-scale warehouse environments.
\end{abstract}

\maketitle


\section{Introduction}

The exponential growth of the e-commerce sector has necessitated significant advancements in warehouse management strategies. Increasing consumer demand for expedited order fulfillment has placed pressure on logistics operations, and major retailers like Amazon, H\&M, and Walmart have made substantial investments in automation and AI-driven warehouse management systems (WMS) to enhance efficiency. One of the primary cost drivers in warehouse operations is order picking, accounting for approximately 60\% of total warehouse expenses \cite{kordos2020optimization}. It directly impacts labor costs, order fulfillment speed, and overall supply chain performance. A well-optimized picking strategy minimizes travel distance while maintaining flexibility to accommodate variations in order composition over time. Since order-picking routes can be represented as graphs, where nodes correspond to product locations, determining an optimal route for each order is a combinatorial optimization problem. To further enhance efficiency, order clustering strategies based on product location and demand frequency have been proposed, allowing frequently co-ordered products to be relocated closer together, thereby reducing long-term travel distances~\cite{mirzaei2021impact}.

Shortest path algorithms play a key role in optimizing order picking. While Dijkstra’s algorithm~\cite{dijkstra1959note} is a well-known approach, its sequential nature limits its efficiency for large-scale problems. The Bellman-Ford algorithm \cite{bellman1958routing,ford1956network}, in contrast, is better suited for parallel computing \cite{hajela2014parallel} and has been successfully implemented on GPUs to accelerate shortest path computations \cite{busato2015efficient}. However, in real-world warehouse settings, order picking involves combinatorial optimization, where factorial growth in possible routes makes exhaustive evaluation impractical. In fact, finding a shortest path for picking operations is a version of the traveling salesman problem (TSP). Since the TSP is NP-complete \cite{cormen2022introduction}, approximate methods are often necessary in practice.

In addition, warehouse managers frequently encounter unpredictable demand patterns and inventory fluctuations, challenges that static optimization methods struggle to address effectively. Instead, stochastic modeling approaches have been shown to significantly enhance logistics decision-making, particularly regarding inventory positioning and order fulfillment \cite{govindan2015reverse}. This is exemplified in Amazon's fulfillment centers where stochastic modeling is incorporated in their robotic picking algorithm to balance pick efficiency and order cycle times. While clustering in the traditional data science sense isn't explicitly used, the algorithm conceptually groups orders based on various criteria to optimize efficiency \cite{allgor2023algorithm}.

Derived from the analysis above, this study proposes an advanced warehouse optimization framework that integrates graph data\-bases, unsupervised clustering, and GPU-accelerated route optimization with the Bellman-Ford algorithm. Unlike conventional WMS that rely on static zone allocations, our approach continuously repositions products based on demand fluctuations. This research extends prior work in warehouse route optimization and clustering methodologies \cite{silva2020integrating,bengtsson2022proposed} by employing a random dynamical systems (RDS) framework that provides a rigorous mathematical foundation for analyzing storage and retrieval efficiencies and modeling uncertainty and variability. By leveraging GPU-based computing, our system can handle large-scale warehouses and high order volumes more effectively.

The system employs clustering on two levels: (1) Order clustering, which reorganizes storage layouts to concentrate frequently co-ordered items and minimize long-term travel distance, and (2) Picking node clustering, which reduces computational complexity in route optimization while maintaining near-optimal results. These techniques are analyzed mathematically in Sections \ref{sec:model} and \ref{sec:routeoptimization} and evaluated through numerical experiments in Sections \ref{sec:numericalexperiments} and . By continuously adapting storage configurations and applying combinatorial optimization strategies, this approach balances short-term efficiency gains with long-term structural improvements, ensuring that warehouse operations remain effective even amid fluctuating order patterns.

To describe the implementation of our system in full detail we develop a mathematical model found in Section \ref{sec:model}. The approach can be summarized as follows:

Products in the warehouse are organized using a graph database, which allows for real-time tracking of their positions. This system ensures that the location of each product is always known, facilitating efficient management of stock levels and movement. Incoming orders are clustered based on their similarity and proximity, grouping orders that share common product types or are located near each other within the warehouse. This clustering is also taken into account for calculating picking routes, which minimizes the travel distance for picking multiple orders, thereby increasing efficiency.

For each cluster of orders, a cluster center is determined. This center represents the optimal location within the warehouse for storing products that are frequently included in the cluster's orders. The calculation of the cluster center takes into account the positions of all products in the cluster and aims to minimize the overall distance required for picking. When a product position is emptied during the picking process, the system triggers a reorganization of the warehouse. The product is restocked at a new position that is closer to the relevant cluster center, ensuring that products are always stored in positions that reduce the travel distance for future picking operations.

\subsection{Related work}\label{ss:relatedworks}

A central principle in classical storage location assignment problem (SLAP) is the \emph{ABC classification} of inventory items according to turnover or demand frequency. In this approach, products are divided into a small number of categories---typically three---such that high-demand (A) items are placed in easily accessible locations, while medium- and low-demand (B and C) items occupy progressively less strategic zones. This strategy has long been recognized as effective for improving picker travel efficiency in picker-to-parts systems~\cite{mirzaei2021impact}.

In our study, we adopt $K=3$ clusters to reflect this well-established ABC structure. The intent is not to optimize the number of clusters per se, but to model how a warehouse that follows an ABC-type policy can reorganize its inventory dynamically as demand patterns evolve. The clustering thus serves as an operational analog of ABC zoning under dynamic conditions rather than as a data-driven search for an optimal $K$. For this purpose, we employ the $K$-means algorithm due to its simplicity, interpretability, and computational efficiency. More sophisticated clustering methods (e.g., density-based or hierarchical approaches) could in principle be used, but the aim here is to demonstrate the feasibility and benefits of cluster-driven reorganization rather than to compare clustering techniques. Future work may consider alternative clustering methods within the same framework to explore their potential impact on convergence speed and stability.

Classical research on order picking, as reviewed by de Koster et al.~\cite{de2007design}, highlights the importance of labor and travel-time considerations in warehouse performance. In contrast, the present work abstracts from such operational constraints---including labor assignment, congestion, and time scheduling---to focus on the structural dynamics of storage-location assignment and product reorganization. Our model assumes unrestricted relocation and omits explicit labor capacity limits, allowing us to isolate and analyze the structural effects of clustering on order consolidation and route optimization. Extending the framework to incorporate labor constraints or relocation delays would represent a natural direction for future research.

While classical studies of picker-to-parts systems emphasize operational batching---where multiple customer orders are grouped to minimize picker travel distance \cite{de2007design}---our approach performs a structural analog of this process. By clustering orders and their associated article types based on spatial and demand similarity, the model effectively performs a dynamic, data-driven form of batching at the storage-location level. This allows reorganization of stock positions that facilitates future co-picking, extending the traditional batching concept from route planning to layout adaptation.

Recent research has also highlighted a growing interest in adaptive storage assignment policies. Rizqi and Chou~\cite{rizqi2025dynamic} integrate a mathematical optimization model with discrete-event simulation, using the former to propose storage and picking strategies and the latter to evaluate performance under realistic warehouse dynamics. Other approaches have focused on reinforcement learning: Teck~et~al.~\cite{teck2025deep} employ deep reinforcement learning (proximal policy optimization, PPO) to optimize rack storage and routing in parts-to-picker systems, while Yerlikaya and Ar{\i}kan~\cite{yerlikaya2024novel} introduce a multi-criteria class-based storage framework that adapts dynamically to order profiles. In contrast to these simulation- and RL-based frameworks, our contribution provides a tractable mathematical model for picker-to-parts systems that incorporates dynamic product reorganization through clustering. Performance is validated through large-scale numerical experiments supported by GPU-based computation, demonstrating the scalability and stability of the approach.

From a computational perspective, our GPU implementation of the routing component uses the Bellman--Ford algorithm for single-source shortest-path computation (SSSP). While alternative GPU-oriented algorithms such as $\Delta$-stepping can achieve higher asymptotic performance through fine-grained load balancing and reduced synchronization overhead~\cite{busato2015efficient}, Bellman--Ford offers superior simplicity and generality. Its straightforward parallelization in CUDA enables transparent and reproducible implementation without reliance on external graph-processing libraries. Since routing constitutes only one step within the integrated clustering--RDS framework, Bellman--Ford provides an appropriate and stable baseline, with more specialized GPU SSSP methods left for future large-scale extensions.

Finally, the broader evolution of warehouse research toward data-driven and cyber-physical systems has been comprehensively reviewed by Boysen et al.~\cite{boysen2019warehousing}. They highlight how digitalization, sensor data, and real-time analytics enable adaptive storage assignment and control strategies, often under the umbrella of ``Warehousing 4.0.'' Our framework aligns with this development by integrating clustering-based reorganization and GPU-accelerated routing into a mathematically transparent model, bridging traditional optimization formulations with scalable, data-oriented computation.

\section{Mathematical model}\label{sec:model}

We describe our warehouse optimization method and our numerical experiments using dynamical systems theory. 
We first introduce the state space (describing the different states of the warehouse) and give the definition of orders.

\subsection{The state space}

We consider a situation where $N$ types of articles $a_1,a_2,\ldots,a_N$, are stored in a warehouse at nodes arranged in three-dimensional space representing racks with shelves at different heights. Each shelf on a rack constitutes its own node, and each node is identified with an element in a three-dimensional array having integer coordinates
\begin{equation}\label{eq:dim}
(i,j,k)\in [ 1,n_x] \times [1,n_y]\times [1,n_z]\subset\N^3,
\end{equation}
where $\N=\{0,1,2,\ldots\}$ is the set of natural number, and $n_x\times n_y\times n_z$ is the total number of nodes in the warehouse.  We assume there is at most one article type $a_n$ at each node (some nodes being empty), and identify each article type $a_n$ with the corresponding ordinal $n\in[1,N]$. Empty nodes are assigned article type $0$.

For each article type $n\in[1,N]$, let $M_n$ be the maximum number of parcels that can be stored at an individual node. We assume this number is independent of which node is being considered; it only depends on what type of article is stored at that node. 
Introduce the space $\mathcal A=[0,N]^{n_x\times n_y\times n_z}$ of three-dimensional arrays. Each array $A=(a_{ijk})_{ijk}\in \mathcal A$ corresponds to one particular choice of how to distribute article types in the warehouse. (The array with only zeros represents an empty warehouse with no article types, and the array with only ones is a warehouse that only contains articles of type $a_1$.) Each array element $a_{ijk}$ is thus an integer $n$ between $0$ and $N$, representing the fact that the node at $(i,j,k)$ contains article type $a_n$.

In other words, for each fixed $A=(a_{ijk})_{ijk}\in \mathcal A$, the node $(i,j,k)$ contains articles of type $a_{ijk}\in[0,N]$. As described above, the maximum number of such articles at one node is $M_{a_{ijk}}$. For each $A=(a_{ijk})_{ijk}\in \mathcal A$ we now consider the fiber over $A$,
$$
F_A\mathcal A=\{M:M=(m_{ijk})_{ijk},\ m_{ijk}\in[0,M_{a_{ijk}}]\}.
$$
This describes all the possible inventory balances $M$ (number of parcels of each article type) for a given article type distribution $A$.
The state space is then the fibration defined by the (disjoint) union
$$
X=\bigcup_{A\in\mathcal A} \{A\}\times F_A\mathcal A=\{(A,M):A\in\mathcal A,\ M\in F_A\mathcal A\},
$$
where $A$ varies over all the possible article type distributions. States in $X$ will be denoted by $x$, so that $x=(A,M)\in X$. We may of course also represent $X$ by attaching to each article type the corresponding inventory balance, so that 
$$
X\cong\{(a_{ijk},m_{ijk})_{ijk}: a_{ijk}\in\mathcal [0,N],\ m_{ijk}\in[0,M_{a_{ijk}}]\},
$$
and we will not distinguish between these two representations.

\begin{rmk}
In the description above we assume that nodes have integer coordinates and only allow one article type per node. This is a deliberate simplification made for clarity of presentation and ease of interpretation. At the cost of additional bookkeeping, the model can be generalized in two natural ways: to allow nodes with real-valued coordinates, and to allow multiple article types per node (commonly referred to as stock keeping units, or SKUs, in the warehouse literature).

First, multiple SKUs may be assigned to the same node by introducing an additional shelf index $l$, so that
\[
(i,j,k,l)\in[1,n_x]\times[1,n_y]\times[1,n_z]\times[1,n_w]\subset\mathbb{N}^4,
\]
where $n_w$ denotes shelf capacity. Allocation then corresponds to a standard bin-packing problem, or to a rectangle-packing problem if shelves are treated as two-dimensional footprints~\cite{coffman1984approximation,wascher2007improved,gue2006effects,petersen2004comparison}. 

Second, node positions may be generalized to real coordinates by introducing a mapping
\[
f_{\mathrm{coord}} : \mathbb{N}^3 \to \mathbb{R}^3, \qquad 
(i,j,k) \mapsto \bigl(x(i,j,k),y(i,j,k),z(i,j,k)\bigr),
\]
so that route optimization is performed on the resulting continuous layout rather than the integer lattice in~\eqref{eq:dim}. 

While these extensions would make the model more realistic, they are beyond the scope of this paper. Our focus here is on the clustering of orders and picking nodes, and the simplified setting suffices for this purpose. More general formulations, such as multi-SKU nodes or continuous coordinates, are natural directions for future work.
\end{rmk}

\subsection{Orders}

An order $o$ is a vector 
$$
o=(m_1,\ldots,m_N)\in\N^N
$$
where $m_n\in\N$ is the number of parcels of article $a_n$ contained in the order. (Here we shall by an order mean any order---individual customer orders or a collection of customer orders---that has been decided should be picked from the warehouse.) Given a state $x=(A,M)$ and an order $o$, the Order Processing algorithm (defined in Algorithm \ref{alg:process_order} in \ref{sec:WMS} below) and a route optimization algorithm (which involves the proposed clustering step based on similarity and proximity, see \S\ref{ss:routes}) provide the nodes from which the order should be collected, along with the picking route between these nodes. The first part is described by a map $w$ given by
\begin{equation}\label{eq:w}
w:(x,o)\mapsto w(x,o)=(w_{ijk})_{ijk}\in\N^{n_x\times n_y\times n_z},
\end{equation}
where $w_{ijk}$ is the number of parcels that should be collected from node $(i,j,k)$, and the total number of parcels in the order is $\Sigma_{i,j,k} w_{ijk}=m_1+\ldots+m_N$. The Update Stock algorithm (Algorithm \ref{alg:update_stock}) then calculates the intermediate state
\begin{equation}\label{eq:faux}
(A,M')=(A,M-w(x,o)).
\end{equation}
The Check Stock algorithm (Algorithm \ref{alg:check_stock}) verifies if there is sufficient stock to fill the order without picking the last parcel on a shelf, i.e., whether $m_{ijk}> w_{ijk}$ for all $(i,j,k)$. If this is not the case, i.e., if $m_{ijk}\le w_{ijk}$ for some $(i,j,k)$ such that $w_{ijk}\ge1$, then the Check Stock and Move Articles algorithm of the warehouse optimization routine (Algorithm \ref{alg:check_and_move}) computes a new warehouse location $(i',j',k')$ at which to restock article $a_{ijk}$ (collecting from the new location after emptying the balance at the old location), and the article array $A$ is updated to $A'$ in which $a_{ijk}$ is replaced by $0$, and $a_{i'j'k'}$ is changed from $0$ to $a_{ijk}$. This gives an updated state
\begin{equation}\label{eq:f}
x'=f(x,o):=(A',M'),
\end{equation}
which is the new state of the warehouse after having completed picking of the order $o$.
If no nodes become empty when filling the order then the intermediate state \eqref{eq:faux} is taken as the the new state $x'$ in \eqref{eq:f}.

\subsection{Random dynamical systems}\label{ss:rds}
To capture the stochastic evolution of warehouse states under repeated picking and resorting, we formulate the system as a random dynamical system (RDS), which provides a compact mathematical framework for modeling how random demand patterns drive successive reorganizations of stored articles over time. Standard mathematical references are Arnold \cite{Arnold} and Strogatz \cite{strogatz2024nonlinear}.

We use discrete time $n\in\N$, where $\N=\{0,1,2,\ldots\}$ as before is the set of natural numbers.
Let $(\Omega,\mathcal F,\mathbb P)$ be a probability space. We take $\Omega$ to be the shift space consisting of sequences $\omega=(o_1,o_2,\ldots)$ of orders $o_n$, and let $\theta:\Omega\to\Omega$ to be the shift operator $\theta\omega=(o_2,o_3,\ldots)$.
This models a sequence of orders $\omega\in\Omega$, placed with probability $\mathbb P$, coming in to the warehouse to be filled one after the other, with $\theta$ driving the flow of orders.
Let $X$ be the state space describing different states of the warehouse. Let $F:\omega\mapsto F_\omega$ be a random variable from $\Omega$ to the set of maps on $X$, i.e., for all $\omega\in\Omega$, $F_\omega$ is a map $F_\omega:X\to X$.
Consider the {\it skew product}
defined by $\Theta(\omega,x)=(\theta\omega,F_\omega(x))$. Iterating $\Theta$ gives
$$
\Theta^n(\omega,x)=(\theta^n\omega,F_{\theta^{n-1}\omega}\circ F_{\theta^{n-2}\omega}\circ \cdots\circ F_\omega(x))
$$
for $n\ge1$, where $\circ$ denotes composition of functions. The second coordinate is a {\it random dynamical system over $\theta$.}

We let $f(x,o)$ be the function in \eqref{eq:f} describing how the state $x$ evolves to $x'=f(x,o)$ by filling the order $o$ once, then, if necessary, restocking articles at new locations as determined by the optimization routine. With $\omega=(o_1,o_2,\ldots)$ we then let $F_\omega(x)=f(x,o_1)$, so that $F_{\theta^{n-1}\omega}(x)=f(x,o_n)$ for $n\ge1$.
The simplest example is when $\omega=(o,o,\ldots)$, corresponding to the same order $o$ being placed over and over (see Experiment 1 in \S\ref{sec:numericalexperiments} below).
To account for small differences in individual orders we will also consider the random example when $\omega=(o_1,o_2,\ldots)$ and the orders $o_n$ are independent, identically distributed perturbations of some fixed order $o$ (see Experiments 2 and 3 in \S\ref{sec:numericalexperiments}).
In this case $\Theta^n(\omega,x_0)$ describes how the (initial) warehouse state $x_0$ evolves to
\begin{equation}\label{eq:xn}
x_n=F_{\theta^{n-1}\omega}\circ\cdots\circ F_{\omega}(x_0)=f(\bullet,o_n)\circ\cdots\circ f(\bullet,o_1)(x_0)
\end{equation}
by filling the first $n$ orders in $\omega$, having potentially restocked articles after each order at new locations as determined by the optimization routine. $\Theta^n(\omega,x_0)$ also contains the information that the next order to be filled is the first entry in $\theta^n\omega=(o_{n+1},o_{n+2},\ldots)$, that is, $o_{n+1}$. We also note that in this case, the random difference equation
$$
x_n=F_{\theta^{n-1}\omega}(x_{n-1}),\quad n\ge1, \quad x_0\in X,
$$
generates a Markov chain $(x_n)_n$ with transitional probability
$P(x,B)=\mathbb P\{\omega:F_{\omega}(x)\in B\}$, cf.~\cite{Kifer}.
Iteration of the picking and restocking process for a sequence of orders $o_1,o_2,\ldots$ in this way is handled by the Main Order Processing algorithm (Algorithm \ref{alg:main}).

\subsection{Routes and clusters}\label{ss:routes}

Given a state $x$ and an order $o$, a route is an ordering of the elements in the set of picking nodes
$$
S_{\rm pick}=S_{\rm pick}(x,o)=\{(i,j,k):w_{ijk}\ne0\}
$$
determined by \eqref{eq:w}. Routes are determined so that picking is completed from nodes on the same rack first (i.e., nodes that have the same $(i,j)$-coordinates) before moving on to the next rack. The stops along the route can then be described as a given ordering of the elements in the set of stopping nodes
\begin{equation*}
S_{\rm stop}=S_{\rm stop}(x,o)=\{(i,j):\text{ there exists $k$ such that }(i,j,k)\in S_{\rm pick}\},
\end{equation*}
and the number of stops is
\begin{equation}\label{eq:stops}
n=\#S_{\rm stop}.
\end{equation}
At stopping nodes with more than one picking node, we can pre-determine the order in which picking is carried out---e.g., top to bottom---so this does not have to not included in the route optimization. We will therefore not distinguish between picking routes and routes between stopping nodes.

\begin{figure}
\begin{center}
\includegraphics[width=0.98\textwidth]{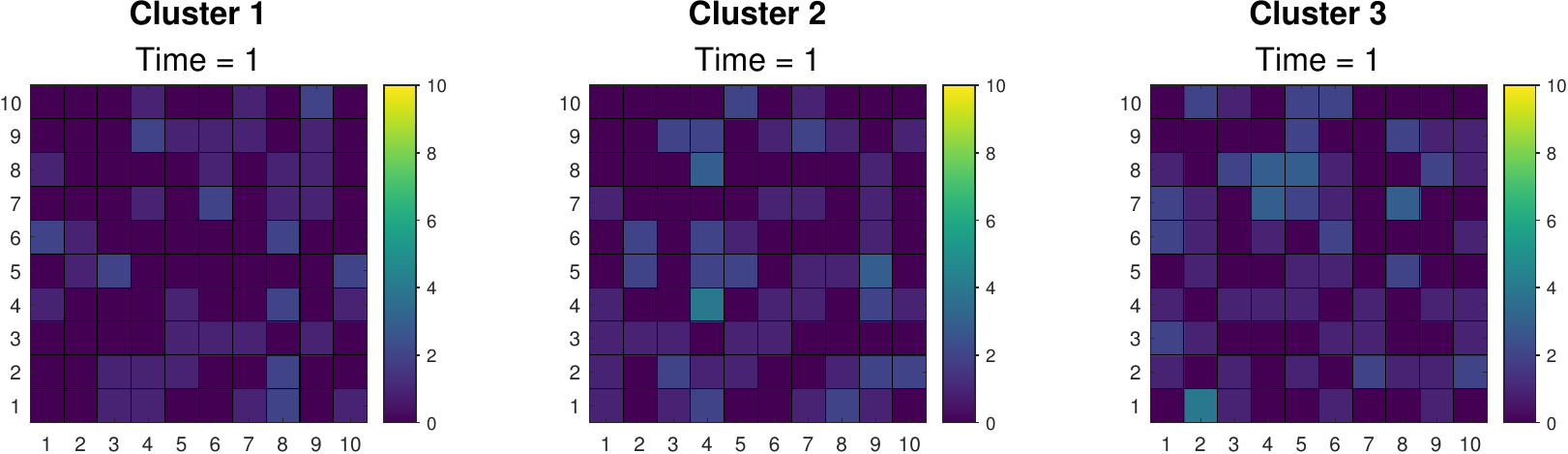}

\vspace{4mm}

\includegraphics[width=0.98\textwidth]{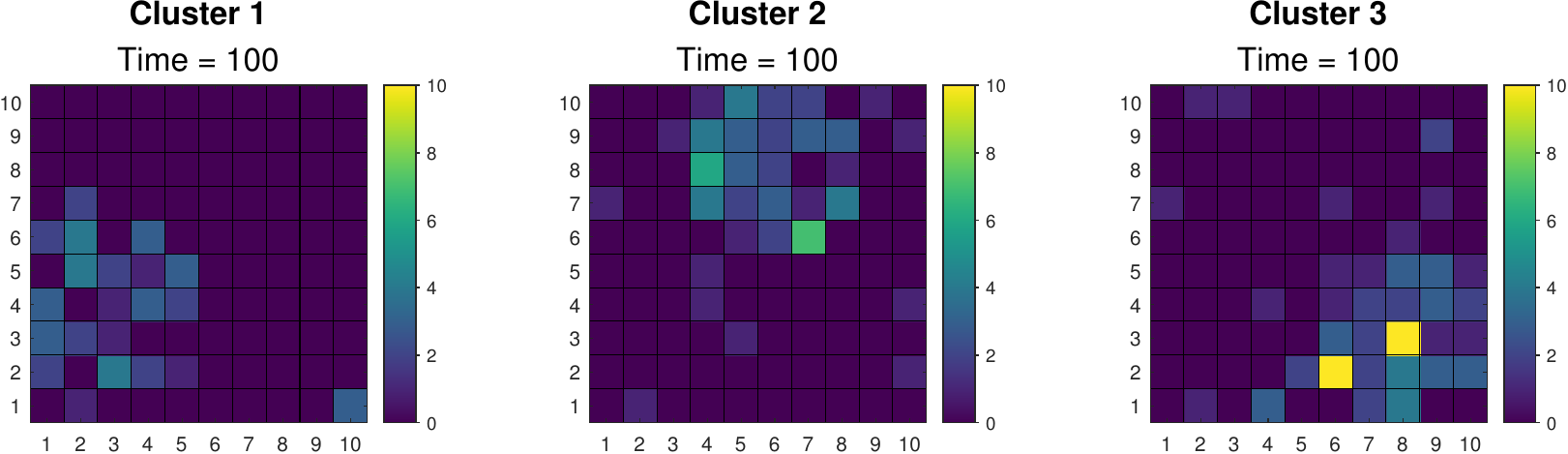}
\caption{Each of the top panels show one of three clusters resulting from $K$-means clustering applied to an order with 200 article types given a randomized initial warehouse state $x_0$. The nodes in the warehouse have coordinates $(i,j,k)$ with $i,j,k=1,\ldots,10$. Stopping nodes are indicated with color coding describing how many picking nodes a given stopping node contains. Each of the bottom panels show one of three clusters resulting from $K$-means clustering applied to the same order being picked for the 100:th time, that is, given the warehouse state $x_{99}$ as in \eqref{eq:xn}, obtained after 99 picking iterations. In the top panels, all three clusters are spread out over the entire warehouse with much overlap between clusters. In contrast, the clusters in the bottom panels are concentrated and much better separated. \label{fig1and2}}
\end{center}
\end{figure}

As explained in the introduction, we propose to group orders into clusters based on similarity and proximity to reduce computational complexity while still optimizing picking. The separation and concentration of the clusters is improved as time progresses by taking clustering into account when restocking. The algorithm restocks an article at an empty picking node whose $(x,y)$-coordinate (i.e., stopping node) is closest to the cluster center of the corresponding cluster. This reduces the area over which each cluster is spread out. Figure \ref{fig1and2} shows an example of this phenomenon, which we will investigate further in Section \ref{sec:numericalexperiments} by performing several numerical experiments. Each of the top panels in Figure \ref{fig1and2} show one of three clusters resulting from applying $K$-means clustering on a picking order $o$ consisting of 200 article types, given a randomized initial warehouse state $x_0$. The nodes in the warehouse have coordinates $(i,j,k)$ with $i,j,k=1,\ldots,10$. The stopping nodes in each cluster are indicated with color coding describing how many picking nodes there are at each stopping node. The three clusters are all more or less spread out over the entire square. In contrast, the bottom panels show $K$-means clustering when the same order $o$ is going to be picked for the 100:th time; the warehouse state is then $x_{99}$ with $x_n$ as in \eqref{eq:xn}, obtained after applying the resorting algorithm having already picked this order 99 times. Article types have now been restocked at new positions that are closer to the relevant cluster center, and clusters are clearly distinguishable and well separated. 

To make this mathematically precise we introduce the following definition. 
Let $m$ be the number of clusters. 
Given a state $x$ and an order $o$, a cluster is a subset of $S_{\rm pick}$ which we will denote by
$$
S_{\rm pick}(\ell)=S_{\rm pick}(\ell;x,o)\subset S_{\rm pick},\quad 1\le \ell\le m,
$$
such that the $S_{\rm pick}(\ell)$ are mutually disjoint for $1\le \ell\le m$ and
$$
S_{\rm pick}=\bigcup_{\ell=1}^m S_{\rm pick}(\ell).
$$
The stops in cluster $\ell$ are
\begin{equation}\label{eq:clusterstops}
S_{\rm stop}(\ell)=S_{\rm stop}(\ell;x,o)=\{(i,j):\text{ there exists $k$ such that }(i,j,k)\in S_{\rm pick}(\ell)\}.
\end{equation}
Note that the sets $S_{\rm stop}(1),\ldots,S_{\rm stop}(m)$ are not mutually disjoint in general, which means that a rack in the warehouse can belong to more than one cluster. In \eqref{eq:clusterstops}, the number of $k$'s such that $(i,j,k)\in S_{\rm pick}(\ell)$ is the picking frequency at stop $(i,j)$ for cluster $\ell$. By taking picking frequency into account, clusters can be represented in two dimensions as in Figure \ref{fig1and2}, where the frequency is represented by the color coding of each stopping node.

We remark that we define clusters mathematically as a subset of the picking nodes for a given order, as opposed to dividing the picking nodes from a collection of orders into different groups or clusters. But since an order here can refer to the total collection of individual customer orders to be picked at a certain time our definition covers also the second scenario.

Because we will study routes between stopping nodes, we define cluster centers as the expected value of the location of stopping nodes in each cluster, taking picking frequency into account.
The cluster centers are then given for cluster $\ell$ by
$(\bar x(\ell),\bar y(\ell))$ where
$$
\bar x(\ell)=\frac{\sum_{(i,j,k)\in S_{\rm pick}(\ell)} i}{\# S_{\rm pick}(\ell)},\quad \bar y(\ell)=\frac{\sum_{(i,j,k)\in S_{\rm pick}(\ell)} j}{\# S_{\rm pick}(\ell)},
$$
with covariance matrix
\begin{equation}\label{eq:covariancematrix}
\Sigma(\ell)=\begin{pmatrix}\sigma_{xx}^2(\ell) & \sigma_{xy}^2(\ell)\\ \sigma_{xy}^2(\ell)&\sigma_{yy}^2(\ell)\end{pmatrix}.
\end{equation}
Here the variances and covariance are given by
\begin{gather*}
\sigma_{xx}^2(\ell)=\frac{\sum_{(i,j,k)\in S_{\rm pick}(\ell)} (i-\bar x(\ell))^2}{\# S_{\rm pick}(\ell)},
\\
\sigma_{yy}^2(\ell)=\frac{\sum_{(i,j,k)\in S_{\rm pick}(\ell)} (j-\bar y(\ell))^2}{\# S_{\rm pick}(\ell)},
\\
\sigma_{xy}^2(\ell)=\frac{\sum_{(i,j,k)\in S_{\rm pick}(\ell)} (i-\bar x(\ell))(j-\bar y(\ell))}{\# S_{\rm pick}(\ell)}.
\end{gather*}

To study how clusters evolve as we apply our resorting procedure to a sequence of orders $o_1,o_2,\ldots$ we let $\Sigma_{n}(\ell)$ be the covariance matrix of cluster $\ell$ as in \eqref{eq:covariancematrix} defined for a state $x_{n-1}$ and order $o_{n}$, $n\ge1$. Here $x_n$ is given by \eqref{eq:xn}, so that $x_n$ is the warehouse state obtained after filling the first $n$ orders in the sequence $o_1,o_2,\ldots$. This will be used in Section \ref{sec:numericalexperiments} where we illustrate the time evolution of the mean and variance of clusters by plotting bivariate Gaussian probability density functions with corresponding mean and covariance matrices as $n$ increases, see Figures \ref{fig:exp1}, \ref{fig:exp2} and \ref{fig:exp3}.

\begin{rmk}
In the formulation above, relocations are assumed to be unrestricted, allowing any article to move freely between storage positions at each iteration. This simplifying assumption isolates the structural effects of clustering and routing from operational factors such as labour capacity or relocation delays, which are discussed as possible extensions in Section~\ref{sec:discussion}.
\end{rmk}

\section{Route Optimization}\label{sec:routeoptimization}

In addition to clustering orders for reorganizing storage layouts, this paper presents a procedure for optimizing order picking routes by using clustering also within a given picking order. In the language of our mathematical model, we recall that a picking route is a route with total number of stops in the $xy$-plane given by $n$ in \eqref{eq:stops}. If there are more than one picking node at a picking stop (cf.~\S\ref{ss:routes}) we pre-define the picking order for such nodes so this does not have to be included when calculating the route. However, we recall that the frequency of picking at each stop does matter, and that the restocking routine biases toward increased frequency. Over time (i.e., upon iteration of the resorting algorithm) this leads clusters to become more concentrated with smaller area and fewer stops, which in turn leads to a reduction in both travel distance and computational complexity. However, finding an optimal route can be a daunting task even after the number of stops have been reduced in this way, so in this section we discuss a technique to reduce computational complexity while maintaining near optimal results when calculating the route at each iteration of the resorting algorithm. We showcase the technique at the end of Section \ref{sec:numericalexperiments}.

For large orders, where the computational complexity of route optimization can be $O(n!)$, we propose to use a GPU-based implementation of the Bellman-Ford algorithm to compute the route. This approach allows for efficient handling of large datasets by significantly reducing computational time. For clarity of exposition we will assume that the entrance and exit are the same in the warehouse, and that these are not part of the route. Routes will then not be closed curves, and we will treat two routes as identical if they differ only in their travel direction. The number of possible routes between $n$ picking stops (i.e., stopping nodes) is then $\frac12 n!$. To reduce computing costs we propose to use unsupervised machine learning to divide the picking stops into $m$ clusters, then finding optimal routes in each cluster, and finally patching these routes together. This has the possibility of substantially reducing the computational complexity. For simplicity we begin with a naive discussion that doesn't leverage any other sophisticated algorithms to reduce complexity.

\begin{thm}\label{thm:reduction}
Let $n$ be the number of picking stops, and let these be divided into $m$ pairwise disjoint clusters, each containing $n_1,\ldots,n_m$ picking stops, where $n_1+\ldots+n_m=n$. If an undirected non-closed route has been determined for each cluster then the total number of possible global undirected routes connecting them is $m!\times 2^{m-1}$. The total cost of also finding optimal routes in each individual cluster is therefore
$$
m!\times 2^{m-1}+\frac12(n_1!+\ldots+n_m!).
$$
\end{thm}

\begin{proof}
Suppose first that an undirected route with two boundary nodes has been found for each cluster. We determine the computational complexity of finding the best global route going through all clusters along these routes using a combinatorial argument. 
First, a starting cluster is chosen which can be done in $m$ ways. From that cluster, a global starting point should be determined from the pair of that cluster's boundary nodes, which can be done in $2$ ways. Next, a second cluster is chosen which can be done in $m-1$ ways. From that cluster, a starting point connecting with the first cluster's end point should be determined from the second cluster's pair of boundary nodes, which can be done in $2$ ways. Continuing this way we see that there are $m!\times 2^m$ ways to choose a global route with a start and end. Hence, there are half as many global undirected routes, namely $m!\times 2^{m-1}$.

Now, for a cluster with $n_j$ picking stops there are $\frac12n_j!$ ways to choose an undirected route. This is independent of the other clusters. Adding these to the cost of choosing a global route completes the proof.
\end{proof}

\begin{figure}
\begin{center}
\includegraphics[width=0.48\textwidth]{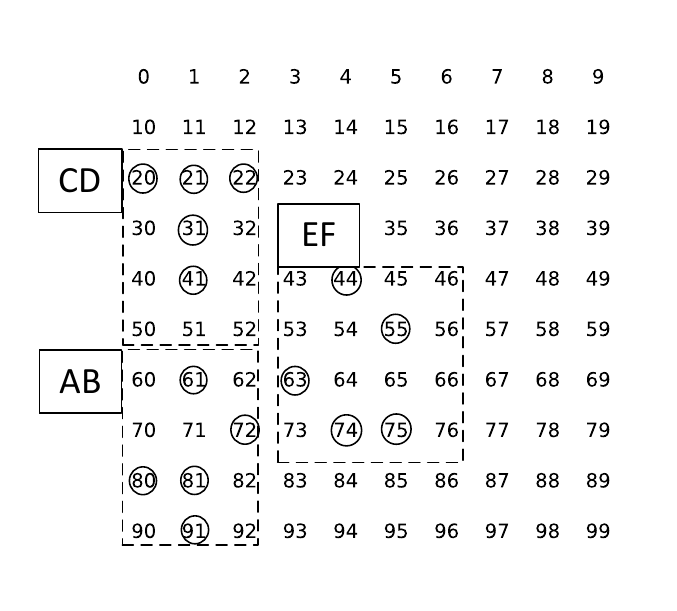}\quad
\includegraphics[width=0.38\textwidth]{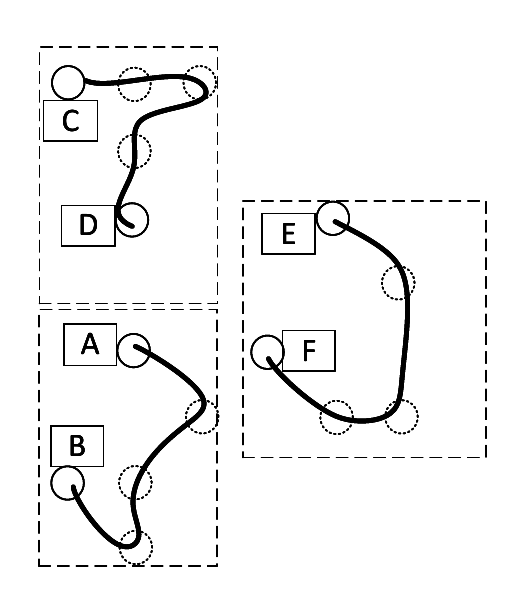}
\caption{An order with 15 stopping positions clustered into 3 regions (left), and the reduced routes for each cluster represented by two boundary nodes (right). For ease of presentation, stopping nodes are labeled from 0 to 99, where node 0 corresponds to $(i,j) = (1,1)$ in Figure \ref{fig1and2}.\label{fig1and2a}}
\end{center}
\end{figure}

The reduction in computational cost described by the theorem is largest when $n_j\sim (n/m)$ for $j=1,\ldots,m$. To give the reader an idea of the magnitude of the reduction in this case, consider a modest $n=12$. Then $12!\approx 5\times 10^{9}$, while if $n_j=n/m$ for $j=1,\ldots,m$ then 
$$
m!\times 2^{m-1}+\frac12(n_1!+\ldots+n_m!)=\begin{cases} 724& m=2,\\ 60& m=3,\\ 204& m=4.\end{cases}
$$
As can be seen, this approach significantly reduces the computational complexity. Although it does not guarantee an optimal route, it ensures that a route close to the optimal one can be achieved within a reasonable computational time.

The case when $n=15$ and $m=3$ is shown in Figure \ref{fig1and2a}. Then $15!\approx 10^{12}$, while $n_1=n_2=n_3=5$ and
$$
m!\times 2^{m-1}+\frac12(n_1!+\ldots+n_m!)= 204.
$$

Of course, we have been comparing the results of the theorem above to a brute force approach for finding optimal routes without employing sophisticated algorithms to reduce computational complexity. However, using for example the Held \& Karp algorithm \cite{held1962dynamic} reduces the time complexity from $\mathcal O(n!)$ to $\mathcal O(n^2\cdot 2^n)$, which, with $n=12$, is still 3 orders of magnitude higher than the examples above---and this without using a corresponding algorithm to find optimal routes within clusters. Hence, the gains in time complexity from the approximate method are still substantial.

\section{Numerical experiments}\label{sec:numericalexperiments}

Here we perform numerical experiments to evaluate the performance of the resorting algorithm for a warehouse model described in Section \ref{sec:model} with the following parameter settings:
\begin{enumerate}
\item \label{settings:1} The warehouse contains $10\times 10\times 10$ picking nodes located at $(i,j,k)\in [1,10]\times [1,10]\times[1,10]$. This means that $n_x=n_y=n_z=10$ in \eqref{eq:dim}. 
\item \label{settings:2} The number of article types is $N=890$, and we assume that the maximum number of articles at a node is $10$ for all article types, so $M_{a_{ijk}}=10$ for all $(i,j,k)$. Thus, the state space $X$ is 
$$
X=\{(a_{ijk},m_{ijk})_{ijk}: a_{ijk}\in[0,890],\ m_{ijk}\in[0,10]\}.
$$
\item \label{settings:3} The initial warehouse state $x_0$ is randomized with the following constraints: It contains $11$ $xy$-coordinate pairs $(i_0,j_0),(i_1,j_1),\ldots,(i_{10},j_{10})$ such that the nodes $(i_0,j_0,k),\ldots(i_{10},j_{10},k)$ are empty for all $k\in [1,10]$. This corresponds to $11$ empty racks, for a total of $110$ empty nodes. A unique article type $a_{ijk}\in[1,890]$ is assigned to each of the remaining $890$ nodes, and we assume that the initial warehouse state is fully stocked, so the balance at each nonempty node $10$. 
\item \label{settings:4} An order $o$ to be picked is a collection of $20$ individual customer or purchase orders. A purchase order consists of $10$ article types with between $1$ and $10$ parcels of each type. The $20$ purchase orders are mutually disjoint, so $o$ consists of $200$ article types. In each picking iteration, the $20$ purchase orders are divided into three groups using $K$-means clustering based on where the articles are stored in the warehouse at that time (with $K=3$ reflecting the classical ABC demand classification scheme, as discussed in Section~\ref{ss:relatedworks}). Each such group is a cluster. An example can be seen in Figure \ref{fig1and2}.
\end{enumerate}

\subsection{Experiment 1}
We first perform picking and resorting when the same order is placed over and over, that is, we let $\omega=(o,o,\ldots)$ and compute the warehouse state 
$$
x_n=F_{\theta^{n-1}\omega}\circ\cdots\circ F_{\omega}(x_0)=f(o,\bullet)\circ\cdots\circ f(o,\bullet)(x_0),\quad n\ge1,
$$
obtained after filling the order $n$ times, having restocked articles at potentially new locations as determined by the optimization routine. This is depicted in Figure \ref{fig1and2}, and the experiment shows that continued application of the picking and resorting procedure leads to more pronounced clusters with better separation. To analyze this effect, we study the time evolution of the mean and variance of the clusters by plotting bivariate Gaussian probability density functions defined using each cluster's mean and covariance matrix $\Sigma_n(\ell)$ as time $n$ increases (cf.~\S\ref{ss:routes}). Figure \ref{fig:exp1} shows the clustering into three clusters of a randomized first order to be picked from an initial randomized warehouse distribution, represented by probability density functions with covariance matrices $\Sigma_1(1),\Sigma_1(2),\Sigma_1(3)$, together with snapshots of the clusters for the  $10$:th and $100$:th picking iteration of the same order, represented by probability density functions with covariance matrices $\Sigma_n(1),\Sigma_n(2),\Sigma_n(3)$ for $n=10$ and $n=100$, respectively. Each panel shows a superposition of the corresponding probability density functions. The figure shows that the clusters become more pronounced with better separation as time progresses.

\begin{figure}
\begin{center}
\begin{tikzpicture}
  \node[inner sep=0pt] (img1) {\includegraphics[width=0.3\textwidth,trim={0.25cm 0.5cm 1cm 0.5cm},clip]{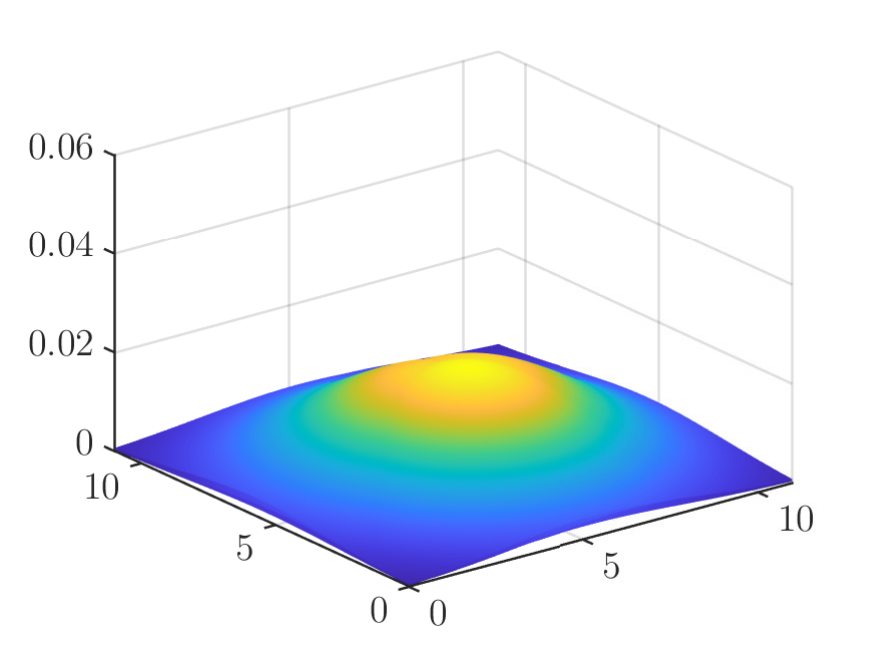}\quad
\includegraphics[width=0.3\textwidth,trim={0.25cm 0.5cm 1cm 0.5cm},clip]{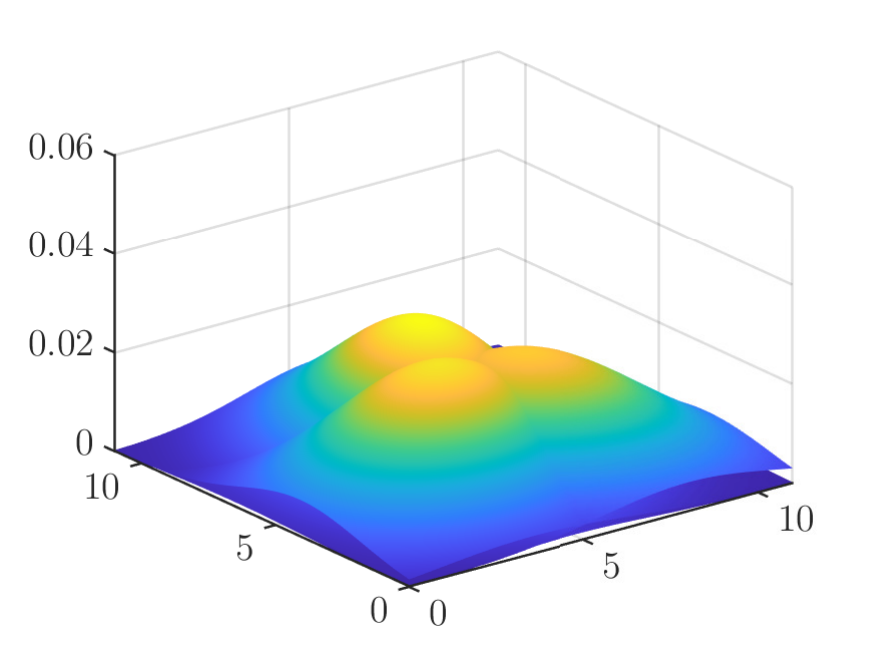}\quad
\includegraphics[width=0.3\textwidth,trim={0.25cm 0.5cm 1cm 0.5cm},clip]{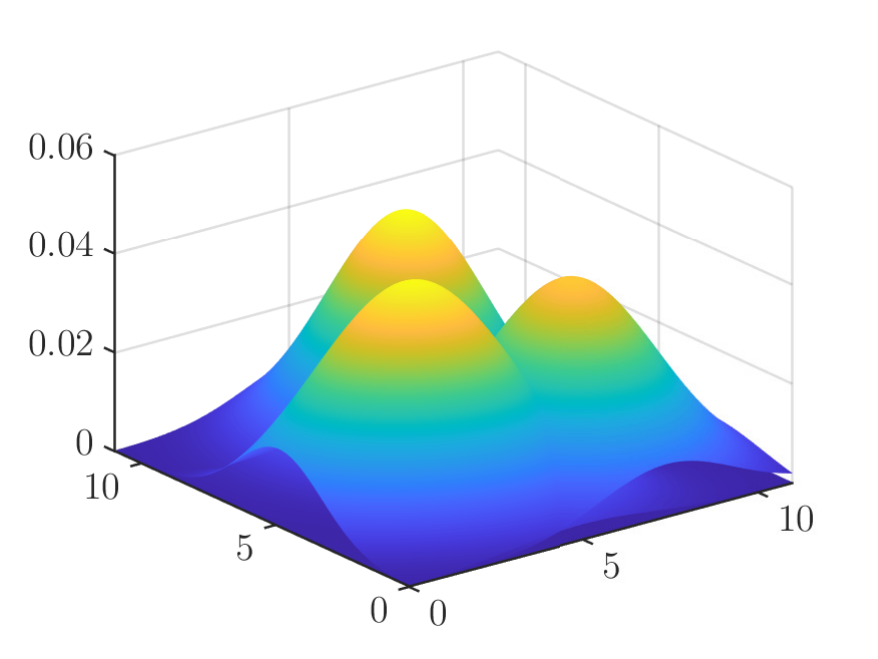}};
\end{tikzpicture}
\end{center}
\caption{Three snapshots of cluster distributions from experiment 1, illustrated by bivariate Gaussian probability density functions defined using each cluster's mean and covariance matrix $\Sigma_n(\ell)$, $\ell=1,2,3$. Each panel shows a superposition of the corresponding probability density function plots. The snapshots are taken at time $n$ (i.e., in the language of \S\ref{ss:routes} they show the clusters computed for the $n$:th order), with $n=1$ (left), $n=10$ (middle), $n=100$ (right). The orders are the same in every iteration. The figure clearly shows that as time progresses, the separation between cluster centers increases, while the bumps become more articulated due to a decreased variance.  A movie showing the evolution of the clusters for $n=1,\ldots,200$ can be found here: \colorurl{magenta}{https://play.hb.se/media/clusters_unperturbed/0_3r1qr1w4}. \label{fig:exp1}}
\end{figure}

\subsection{Experiment 2}
To account for small differences in individual orders, we next consider the random example when $\omega=(o_1,o_2,\ldots)$ and the orders $o_n$ are independent identically distributed perturbations of the order $o$ used in the previous experiment. As described in \S\ref{ss:rds} we then compute the warehouse state after picking the first $n$ orders via
$$
x_n=F_{\theta^{n-1}\omega}\circ\cdots\circ F_{\omega}(x_0)=f(o_n,\bullet)\circ\cdots\circ f(o_1,\bullet)(x_0).
$$
We let each order $o_n$ be constructed as described under item \ref{settings:4} above, but with the first of the $10$ article types in each of the $20$ purchase orders randomly selected from the $890$ article types with uniform distribution. Clearly, two such purchase orders will differ by approximately $10\%$ in article types, and two consecutive picking orders consisting of $20$ such purchase orders that are mutually disjoint to begin with then also differ by approximately $10\%$ in article types, counting multiplicity. The resulting time evolution of clusters when picking and resorting such a sequence of orders $\omega=(o_1,o_2,\ldots)$ is described in Figure \ref{fig:exp2}, and we can see that the resorting procedure still leads to clusters becoming more pronounced with better separation, even in the presence of small variations in orders. Note that since the initial clusters depend on both the initial warehouse state $x_0$ and the first order $o_1$, the initial clusters in this experiment will usually be different from the initial clusters in the previous experiment which were obtained from $x_0$ and the unperturbed order $o$.

\begin{figure}
\begin{center}
\begin{tikzpicture}
  \node[inner sep=0pt] (img1) {\includegraphics[width=0.3\textwidth,trim={0.25cm 0.5cm 1cm 0.5cm},clip]{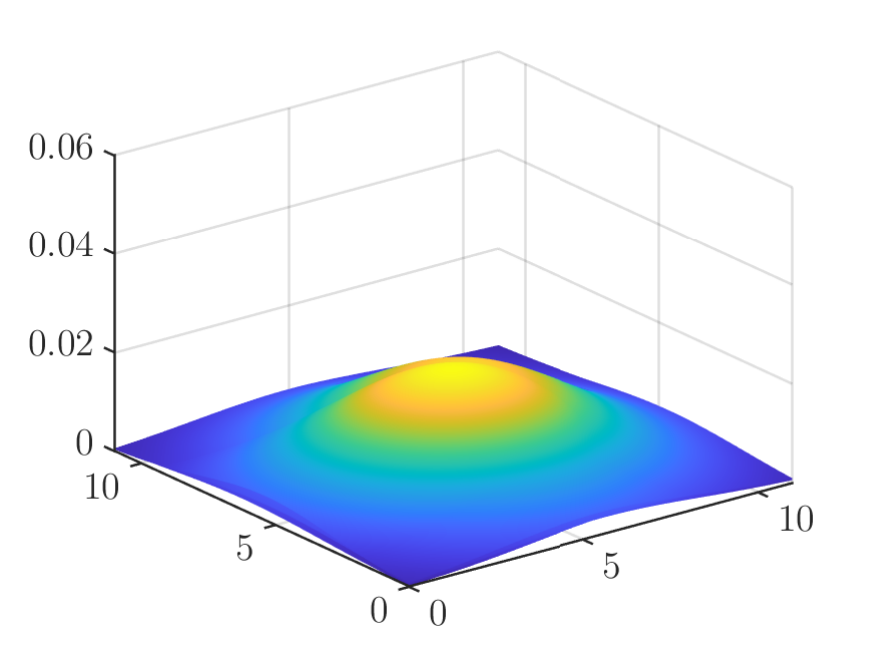}\quad
\includegraphics[width=0.3\textwidth,trim={0.25cm 0.5cm 1cm 0.5cm},clip]{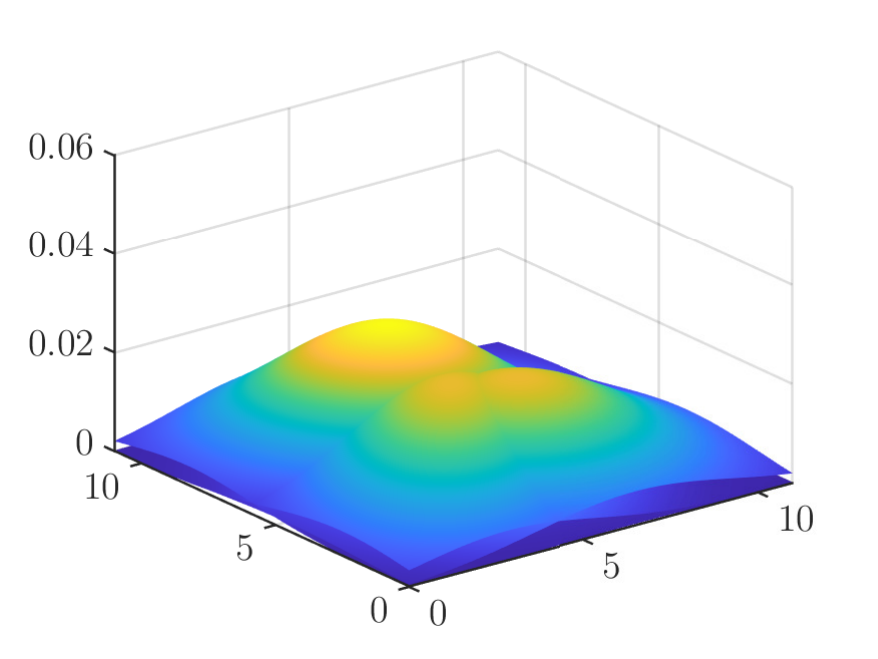}\quad
\includegraphics[width=0.3\textwidth,trim={0.25cm 0.5cm 1cm 0.5cm},clip]{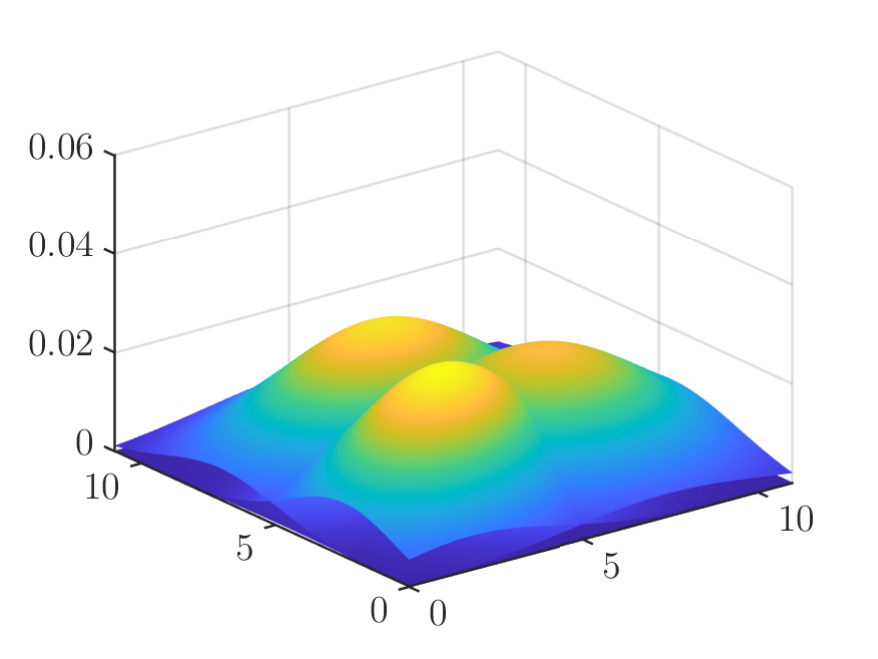}};
 
\end{tikzpicture}
\end{center}
\caption{Three snapshots of cluster distributions from experiment 2, illustrated by bivariate Gaussian probability density functions defined using each cluster's mean and covariance matrix $\Sigma_n(\ell)$, $\ell=1,2,3$. Each panel shows a superposition of the corresponding probability density function plots. The snapshots are taken at time $n=1$ (left), $n=10$ (middle), $n=100$ (right). Two picking orders now differ by approximately $10\%$ on average. Although not as pronounced, we still clearly see an increased separation and decreased variance of the clusters as time progresses. A movie showing the evolution of the clusters for $n=1,\ldots,200$ can be found here: \colorurl{magenta}{https://play.hb.se/media/clusters_static_perturbation/0_t9n5lbme} \label{fig:exp2}}
\end{figure}

\subsection{Experiment 3}
For comparison, we perform the same experiment but instead of always replacing the first of the $10$ article types in each purchase order with a random article type, we replace a randomly selected article type in each purchase order of the original order $o$ with a random article type. Two such purchase orders will then differ on average by around $19\%$ in article types, and two consecutive picking orders consisting of $20$ such purchase orders that are mutually disjoint to begin with also differ by approximately $19\%$ in article types on average, counting multiplicity. The result is described in Figure \ref{fig:exp3} and we can see that although still there, the clusters are now showing less of a tendency of becoming more pronounced with better separation.

\begin{figure}
\begin{center}
\begin{tikzpicture}
  \node[inner sep=0pt] (img1) {\includegraphics[width=0.3\textwidth,trim={0.25cm 0.5cm 1cm 0.5cm},clip]{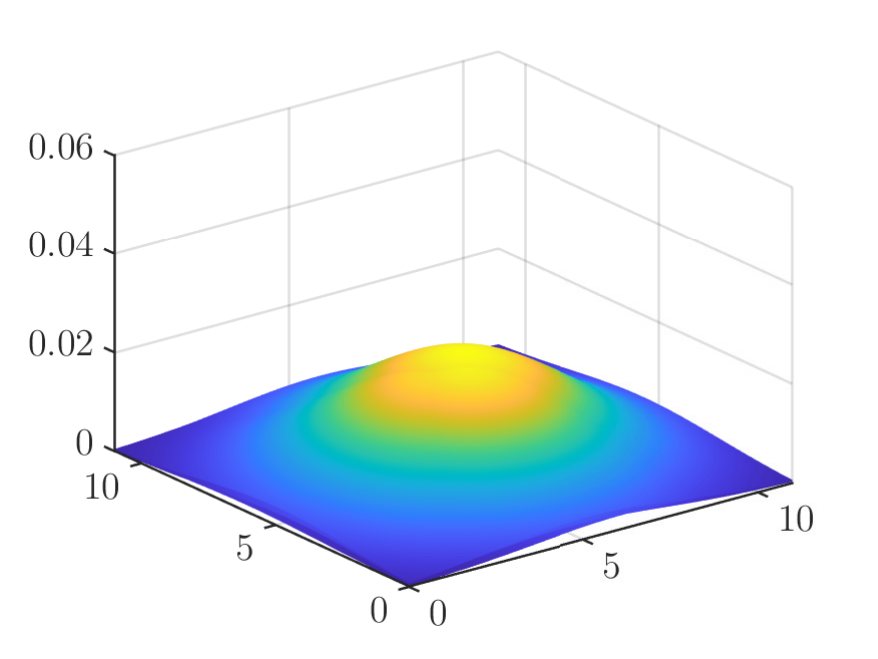}\quad \includegraphics[width=0.3\textwidth,trim={0.25cm 0.5cm 1cm 0.5cm},clip]{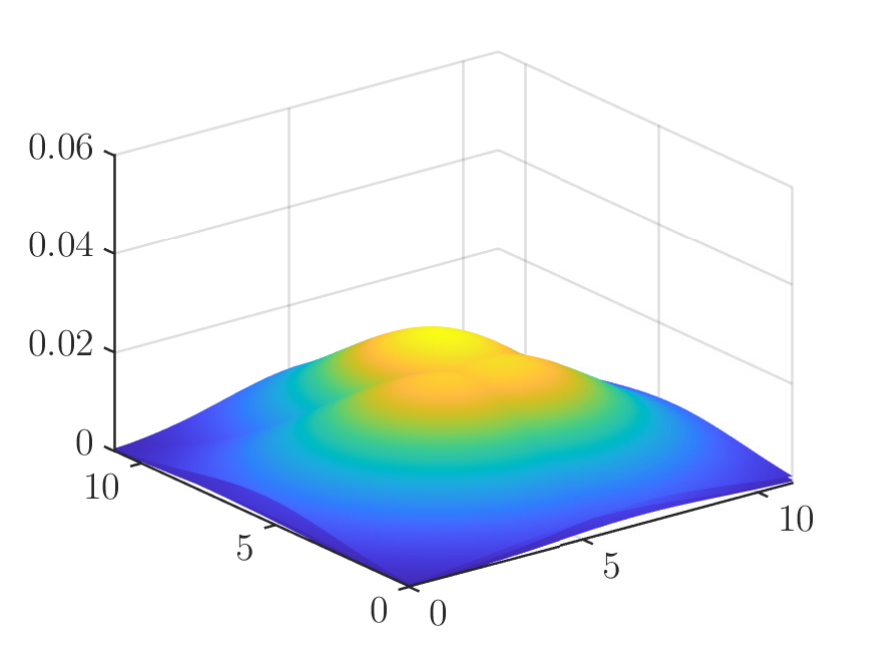}\quad \includegraphics[width=0.3\textwidth,trim={0.25cm 0.5cm 1cm 0.5cm},clip]{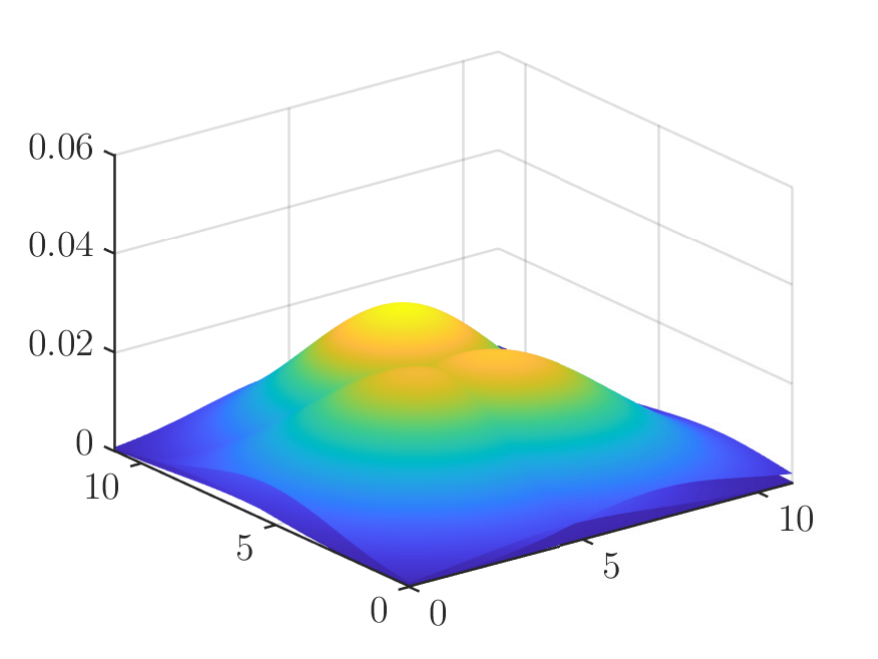}};
\end{tikzpicture}
\end{center}
\caption{Three snapshots of cluster distributions from experiment 3, illustrated by bivariate Gaussian probability density functions defined using each cluster's mean and covariance matrix $\Sigma_n(\ell)$, $\ell=1,2,3$. Each panel shows a superposition of the corresponding probability density function plots. The snapshots are taken at time $n=1$ (left), $n=10$ (middle), $n=100$ (right). Two picking orders now differ by approximately $19\%$ on average, and we only see a slight improvement in the separation and variance of the clusters as time progresses. A movie showing the evolution of the clusters for $n=1,\ldots,200$ can be found here: \colorurl{magenta}{https://play.hb.se/media/clusters_random_perturbation/0_g0zx8ooj} \label{fig:exp3}}
\end{figure}

\subsection{Cluster separation}
As experiments 1, 2, and 3 show, the picking and resorting procedure leads in time to increased separation of the initial clusters, although this separation becomes less pronounced as the variability of the orders increases. To quantify this effect we measure the area of the triangle defined by the cluster center points, and see how this varies with time for the three different setups used in experiments 1, 2, and 3. For experiment 1 the locations of the cluster centers vary fairly slowly, while for experiment 2 and especially 3 the locations vary more rapidly, and as a result the area fluctuates more for experiment 2 and 3. To increase signal strength we therefore average the calculations over 10 runs for each experiment type 1, 2, and 3, and compute the area as a function of time (number of iterations of the picking and resorting algorithm). 

The result is presented in Figure \ref{fig:area}. For all experiment types, the average area increases from an initial average area of around $0.1$; for experiment 1 it has increased to around 6 after the first 20 iterations; for experiment 2 it has increased to around 3 after the first 20 iterations; for experiment 3 it has increased to around 1 after the first 20 iterations. In all cases, the increase is immediate and exceeds one order of magnitude, providing quantitative evidence that the resorting procedure effectively improves cluster separation.

\begin{figure}
	\centering
\begin{tikzpicture}
\begin{axis}[
    y={0.019\linewidth},
    x={0.11\linewidth},
    xtick={0,1,2},
    ytick={0,2,4,6,8,10},
    xticklabels={0,100,200},
xmin=0,
xmax=2,
ymin=0,
ymax=11,
xlabel={Iteration},
ylabel={Area},
ylabel near ticks, 
axis on top=true,
]
\addplot graphics[xmin=0,ymin=0,xmax=2.01,ymax=11]{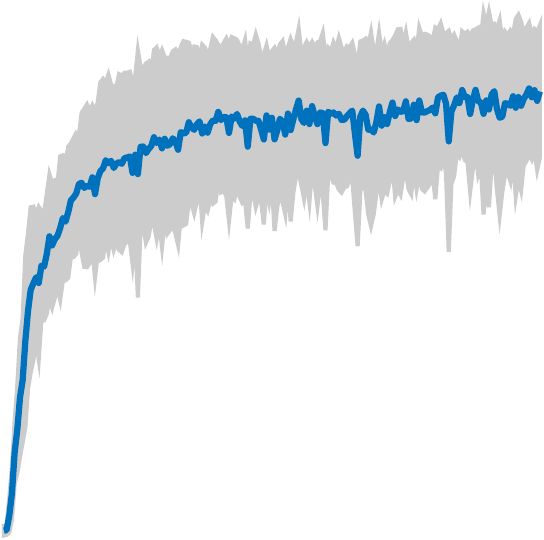};
\end{axis}
\end{tikzpicture}
\quad
\begin{tikzpicture}
\begin{axis}[
    y={0.019\linewidth},
    x={0.11\linewidth},
    xtick={0,1,2},
    ytick={0,2,4,6,8,10},
    xticklabels={0,100,200},
xmin=0,
xmax=2,
ymin=0,
ymax=11,
xlabel={Iteration},
axis on top=true,
]
\addplot graphics[xmin=0,ymin=0,xmax=2.01,ymax=11]{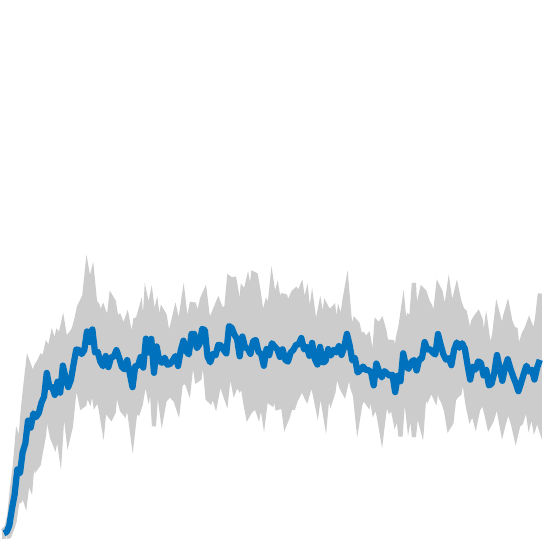};
\end{axis}
\end{tikzpicture}
\quad
\begin{tikzpicture}
\begin{axis}[
    y={0.019\linewidth},
    x={0.11\linewidth},
    xtick={0,1,2},
    ytick={0,2,4,6,8,10},
    xticklabels={0,100,200},
xmin=0,
xmax=2,
ymin=0,
ymax=11,
xlabel={Iteration},
axis on top=true,
]
\addplot graphics[xmin=0,ymin=0,xmax=2.01,ymax=11]{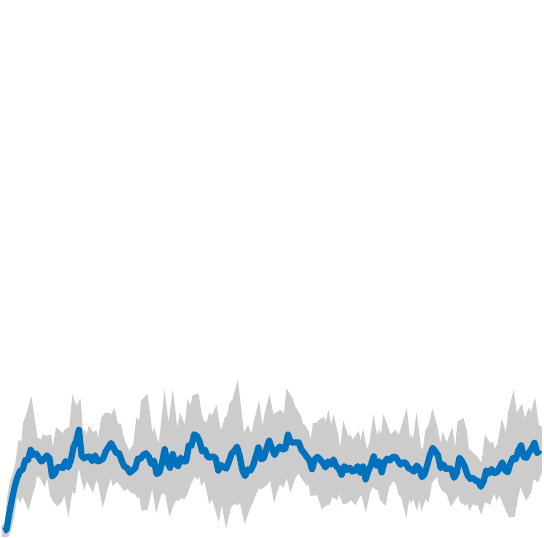};
\end{axis}
\end{tikzpicture}
\caption{Area of the triangle defined by the cluster center points as a function of the number of picking iterations (time) for experiment 1 (left), experiment 2 (middle), and experiment 3 (right). To increase signal strength, the area is averaged over 10 experiments of each type, shown in blue. The gray region indicates the average area plus minus one standard deviation. 
The average area for the initial warehouse state is around $0.1$, and in all three panels this increases substantially within the first 10 or 20 iterations as the clusters become more separated due to the picking and resorting procedure. However, the increase is visibly smaller for experiment 2 and even more so for experiment 3, indicating less pronounced separation of clusters as the noise level increases. 
 \label{fig:area}}
\end{figure}

\subsection{Quantitative analysis of cluster quality}

Here we quantify and compare the increase in cluster concentration and separation for the three experiments using silhouette scores to evaluate cluster quality~\cite{rousseeuw1987silhouettes}. We first recall the definition of silhouette score, so let $i$ denote a data point assigned to a cluster $A$. Define
\[
a(i) = \frac{1}{|A|-1} \sum_{\substack{j \in A \\ j \neq i}} d(i,j),
\]
as the average dissimilarity between $i$ and all other points in its own cluster $A$. Here $d(i,j)$ is the distance between the data points $i$ and $j$.
For every other cluster $C \neq A$, compute the average dissimilarity of $i$ to all points in $C$, and let
\[
b(i) = \min_{C \neq A} \frac{1}{|C|} \sum_{j \in C} d(i,j).
\]
The silhouette value for point $i$ is then defined as
\[
s(i) = \frac{b(i) - a(i)}{\max\{a(i),\, b(i)\}}, \qquad -1 \leq s(i) \leq 1.
\]
The overall silhouette score of a clustering solution with $n$ points is obtained by averaging over all points:
\[
S = \frac{1}{n} \sum_{i=1}^{n} s(i),
\]
with values ranging from $-1$ to $+1$. Higher scores indicate more compact and well-separated clusters, while values near~$0$ indicate overlapping clusters and negative values suggest misclassification. As a rule of thumb, scores above about~$0.5$ reflect strong structure, scores around~$0.25$--$0.5$ indicate moderate clustering, and values close to~$0$ or below denote weak or absent cluster structure.

As in the previous subsection we perform 10 runs for each experiment type 1, 2, and 3 (abbreviated Exp 1, Exp 2, Exp 3 below), and calculate the silhouette score at each iteration. Figure~\ref{fig:silhouette-trajectories} shows the evolution of silhouette scores over 100 iterations for the three experiment types. Exp 1 increases rapidly and stabilizes around 0.45, Exp 2 reaches about 0.10, and Exp 3 increases to around zero. The confidence intervals separate clearly between the experiments.

\begin{figure}
\begin{center}
\begin{tikzpicture}
\node[inner sep=0pt] (img) {\includegraphics[width=0.61\linewidth]{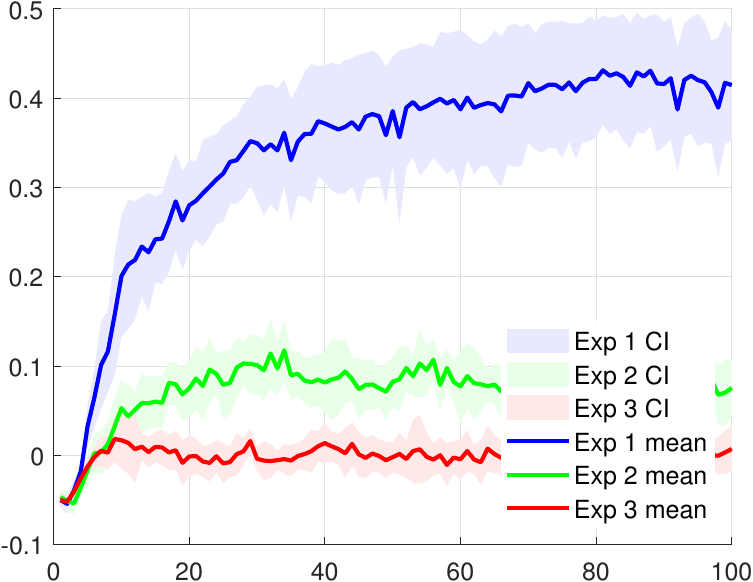}};
\node[rotate=90] at ([xshift=-4mm,yshift=1mm]img.west) {Silhouette score};
\node at ([xshift=2.2mm,yshift=-3mm]img.south) {Iteration};
\node at ([xshift=1.5mm,yshift=2mm]img.north) {\bf Silhouette trajectories with 95\% CI};

\path ([xshift=0mm]img.west) -- ([xshift=6.5mm]img.east) coordinate[pos=0](bbwest) coordinate[pos=1](bbeast);
\useasboundingbox (bbwest |- img.south) rectangle (bbeast |- img.north);
\end{tikzpicture}
\end{center}
\caption{Silhouette score trajectories over 100 iterations for a warehouse with $10 \times 10 \times 10$ nodes. Solid lines show the mean silhouette score across runs, with shaded bands indicating 95\% confidence intervals. Experiment 1 shows strong and stable improvement, experiment 2 reaches modest gains, and experiment 3 remains close to zero, indicating limited clustering performance under high noise. \label{fig:silhouette-trajectories}}
\end{figure}

To quantify these improvements, we compare final to initial silhouette scores within each run. Paired $t$-based confidence intervals indicate substantial improvement in Exp 1 ($\Delta = 0.47$, 95\% CI [0.41, 0.53]), smaller but positive improvement in Exp 2 ($\Delta = 0.12$, 95\% CI [0.09, 0.16]), and only marginal improvement in Exp 3 ($\Delta = 0.06$, 95\% CI [0.03, 0.08]). These results suggest that, for the parameter settings described in the beginning of Section \ref{sec:numericalexperiments}, the method remains effective under low and, to some extent, moderate noise, but yields limited gains under high noise. This is in accordance with the visual evidence of Figures \ref{fig:exp1}--\ref{fig:area}. We note that under these settings the pick list already covers a substantial fraction of warehouse article types, and this fraction grows further with added noise. (As we will see, when the ratio of pick-list to warehouse article types is smaller, the method performs well also under moderate to high noise, see Section \ref{sec:large_experiments}.)

To assess whether the magnitude of improvement differs between experiments, we perform pairwise permutation tests on the run-level improvements (10\,000 resamples). All comparisons are statistically significant after Bonferroni correction (Exp 1 vs.\ Exp 2: $p < 0.0001$; Exp 1 vs.\ Exp 3: $p < 0.0001$; Exp 2 vs.\ Exp 3: $p = 0.0063$). Effect sizes are also large: Exp 1 dominates both Exp 2 and Exp 3 (Cohen's $d \ge 5$, Cliff's $\delta = 1.00$ in both cases), while Exp 2 also outperforms Exp 3 with a large effect (Cohen's $d = 1.5$, Cliff's $\delta = 0.76$). For reference, Cohen's $d$ values of 0.2, 0.5, and 0.8 are often interpreted as small, medium, and large effects, respectively~\cite{cohen2013statistical}; for Cliff's $\delta$, thresholds of 0.147, 0.33, and 0.474 have been suggested for small, medium, and large effects~\cite{cliff1993dominance,romano2006appropriate}. With only three pairwise comparisons, the Bonferroni correction had minimal impact on the results and did not alter the overall conclusions. Thus, although performance degrades with noise, clustering improvements remain detectable even under moderate noise. The final values and statistical comparisons are summarized in Table \ref{table1}.

\begin{table}
\centering
\caption{Silhouette score improvements ($\Delta$) with 95\% confidence intervals, and pairwise statistical comparisons between experiments for a warehouse with $1000$ nodes. Pairwise $p$-values are from permutation tests (Bonferroni corrected). Effect sizes are reported as Cohen's $d$ and Cliff's $\delta$. \label{table1}}
\begin{tabular}{lccc}
\toprule
 & Exp 1 & Exp 2 & Exp 3 \\
\midrule
Mean improvement $\Delta$ & 0.47 & 0.12 & 0.06 \\
95\% CI & [0.41, 0.53] & [0.09, 0.16] & [0.03, 0.08] \\
\midrule
\multicolumn{4}{l}{\textbf{Pairwise comparisons}} \\
\midrule
Exp 1 vs Exp 2 & \multicolumn{3}{l}{$p < 0.0001$, $d = 5.0$, $\delta = 1.0$} \\
Exp 1 vs Exp 3 & \multicolumn{3}{l}{$p < 0.0001$, $d = 6.3$, $\delta = 1.0$} \\
Exp 2 vs Exp 3 & \multicolumn{3}{l}{$p = 0.0063$, $d= 1.5$, $\delta = 0.76$} \\
\bottomrule
\end{tabular}
\end{table}

\subsection{Route optimization}

Here we analyze the effects that the resorting algorithm has on route optimization, and showcase the approximate method of calculating optimal routes that we describe in Section \ref{sec:routeoptimization}, performed over 10 independent runs. Each run starts from an arbitrary customer order, consisting of 10 products with varying quantities. The order is assumed to be recurring (as in Experiment 1 above), and at the initial stage, iteration 1, the articles are randomly distributed throughout the entire domain of the warehouse node positions. Since the customer order is clustered, it will trigger a reorganization of clustered products, causing them to relocate closer to the cluster center.  As seen above, the clusters tend to become more concentrated with reorganization, suggesting that the picking area within each cluster decreases. At iteration 300 we should therefore see a decrease in the total picking distance for each customer order.

We first present one representative trial, choosing the case with the largest deviation from optimal among the 10 independent runs. In this instance, the obtained route length exceeded the optimal by 25\%, which constitutes the worst performance observed under the tested settings. The order for this case is presented in Table \ref{tab:product_quantities}.

\begin{table}
\centering
\caption{Products and their quantities. The position of each article in the warehouse is encoded through a value between $0$ and $99$ to describe its $(x,y)$ coordinate in terms of a flattened $10\times10$ grid starting from bottom left, together with a value between $0$ and $9$ to describe its $z$ coordinate.     \label{tab:product_quantities}}
\begin{tabular}{lc}
\toprule
\textbf{Product} & \textbf{Quantity} \\
\midrule
a17\_4 & 2 \\
a36\_6 & 2 \\
a57\_8 & 6 \\
a9\_7  & 6 \\
a90\_0 & 6 \\
a77\_8 & 7 \\
a93\_6 & 5 \\
a38\_1 & 4 \\
a23\_2 & 5 \\
a24\_4 & 4 \\
\bottomrule
\end{tabular}
\end{table}

When evaluating all possible route permutations through a graph with $n$ vertices, the total number of unique paths is given by $n!$, assuming each vertex is visited exactly once. For each such permutation, $(n - 1)$ transitions must be evaluated to compute the total route cost, resulting in $n!(n - 1)$ operations in total. In the CUDA C implementation (see Subsection \ref{ss:GPUimplementation} for details), each route permutation may be assigned to a separate thread, and within each thread, an internal loop is used to process the transitions between vertices. If $V$ denotes the number of vertices (i.e., the number of stops per route), and $N = n!$ is the total number of permutations, then the computational structure involves operations proportional to $V \times N$. This scaling becomes problematic as $n$ increases, since $n!(n - 1)$ grows factorially and dominates memory consumption. Empirically, we observed that values of $N$ larger than 2,903,040 lead to memory errors in the CUDA kernel, likely due to this exponential increase. As a result, processing large values of $n$ without segmenting the computation or distributing it across multiple kernel launches becomes infeasible due to the limitations in available GPU memory, see Figure \ref{fig1and2b}.

\begin{figure} 
\begin{center}
\includegraphics[width=0.48\textwidth]{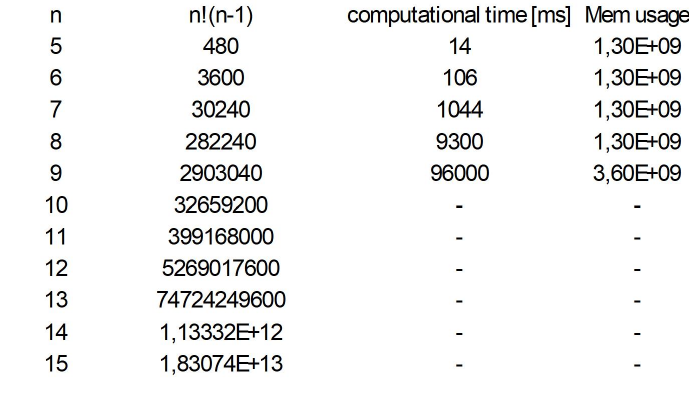}\quad
\includegraphics[width=0.48\textwidth]{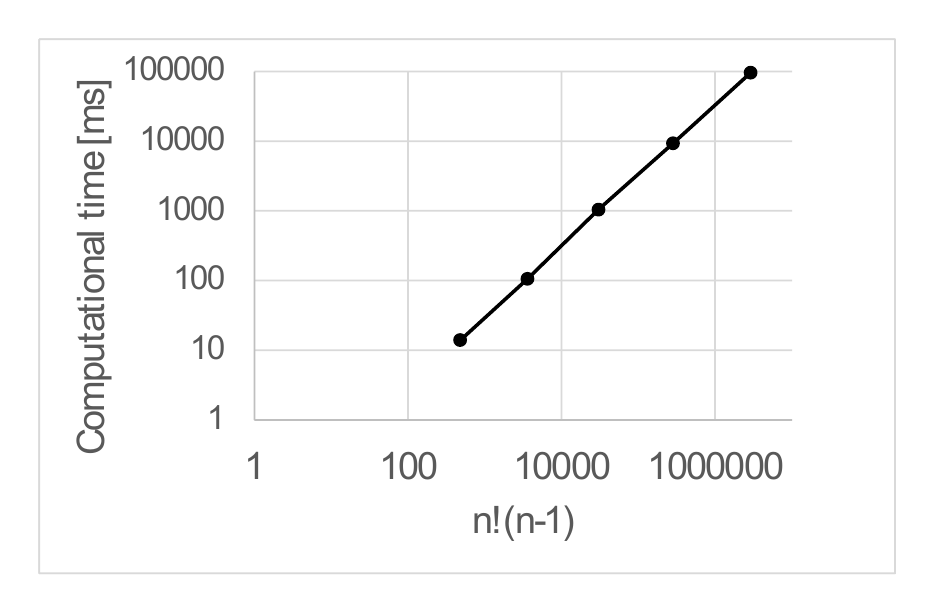}
\caption{Data from the Bellman Ford GPU version (RTX6000, 24GB VRAM) (left), and a comparison between $n!(n-1)$ and computational time (right).\label{fig1and2b}}
\end{center}
\end{figure}

Hence, since the number of stop positions in our experiment is \(n=10\), we use the approximate method for calculating an optimal route described in Section \ref{sec:routeoptimization}. Initially, the number of stops is 10 while the number of passed nodes is 27, and the order is picked over a large portion of the warehouse area. By iteration 300, the clustered order regions have shrunk, and the number of passed nodes is now 15, determined using the approximate method. This corresponds to a 44\% reduction in route length at iteration 300. At this stage, the reorganization has resulted in many articles being located at the same stopping node, reducing the number of combinations to \(6!\). This reduction enables a full combinatorial calculation of the optimal route. The result of the simulation is shown in Figure \ref{route_comp}. The approximate method does not produce the exact same result as the optimal method, with an optimal distance of 12. However, given the challenges with large \(n!\), the approximate method allows for results close to the optimal solution in real-time when using GPUs.

\begin{figure}
\begin{center}
\includegraphics[width=1.0\textwidth]{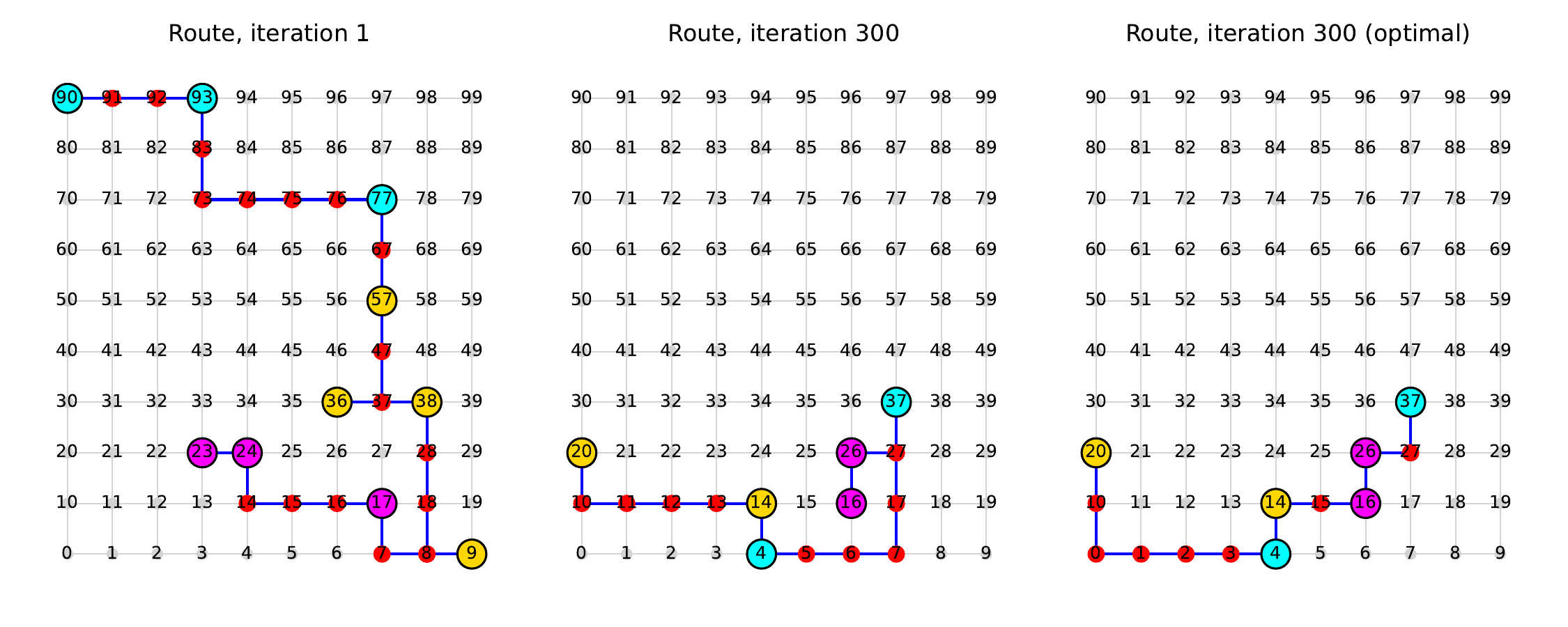}
\caption{Illustrative example of the clustering-based route optimization, showing picking routes for the customer order from Table \ref{tab:product_quantities} at different iterations of the resorting algorithm. This is the trial with the largest deviation from optimal among the 10 runs considered, where the obtained route length exceeds the optimal by 25\% at final iteration. The stopping positions along the route are indicated by the larger nodes, and they are clustered into 3 color coded regions using the methodology from Section \ref{sec:routeoptimization}. At iteration 1 there are 10 stopping nodes, and the left panel shows the computed route. At iteration 300, the resorting algorithm has reduced the number of stopping positions to $6$, and the middle panel  shows the computed approximately optimal route with length 15. For comparison, the right panel shows the optimal route after 300 iterations, with length 12. \label{route_comp}}
\end{center}
\end{figure}

While this example illustrates the worst-performing case, we next summarize results across all runs described in Table \ref{tab:route_results}. The ratio of obtained to optimal route length had a mean of 1.08 (SD = 0.08, median = 1.09), ranging from 1.00 (exact optimal) to 1.25 (worst case). The 95\% confidence interval for the mean ratio was [1.02, 1.14], indicating that the method typically produced solutions within 2--14\% of the optimal route length despite requiring much less computation. No significant correlation was found between optimal route length and approximation ratio ($r = 0.01$, $p = 0.99$), indicating that performance was stable across different route lengths.

\begin{table}
\centering
\caption{Comparison between optimal route length and clustering-based route length across 10 selected cases. Diff = difference between obtained and optimal length. Ratio = obtained/optimal.}
\begin{tabular}{lrrrr}
\hline
Case & Optimal & Obtained & Diff & Ratio \\
\hline
1 & 13 & 15 & 2 & 1.15 \\
2 & 11 & 11 & 0 & 1.00 \\
3 &  8 &  9 & 1 & 1.12 \\
4 &  7 &  7 & 0 & 1.00 \\
5 & 11 & 12 & 1 & 1.09 \\
6 & 12 & 15 & 3 & 1.25 \\
7 & 14 & 14 & 0 & 1.00 \\
8 & 14 & 14 & 0 & 1.00 \\
9 & 10 & 11 & 1 & 1.10 \\
10 & 12 & 13 & 1 & 1.08 \\
\hline
Mean & -- & -- & -- & 1.08 \\
SD   & -- & -- & -- & 0.08 \\
95\% CI & -- & -- & -- & [1.02, 1.14] \\
\hline
\end{tabular}
\label{tab:route_results}
\end{table}

\subsection{GPU Implementation}\label{ss:GPUimplementation}
The Bellman--Ford algorithm was parallelized and implemented in CUDA~C, targeting an NVIDIA RTX~6000 GPU (Turing architecture, 24~GB GDDR6). Each edge relaxation was assigned to one GPU thread, enabling full parallel coverage of all $E$ edges per iteration. The kernel was configured with 256 threads per block, resulting in $\lceil E / 256 \rceil$ blocks per launch. For a representative instance with $V=9716$ vertices and $E=26\,999$ directed edges, this configuration corresponds to 106 thread blocks per iteration. 
The graph data were stored in edge-list format using integer arrays \texttt{u}, \texttt{v}, and \texttt{w} for source vertices, destination vertices, and edge weights, respectively. The distance and predecessor arrays (\texttt{dist}, \texttt{pred}) were allocated in global memory as contiguous vectors of size $V \times N$, where $N$ denotes the number of source vertices processed simultaneously. This layout allows coalesced memory access and efficient indexing across threads. The total device memory footprint for all arrays was approximately 1~MB, which is negligible relative to the available VRAM, confirming that capacity was not a limiting factor. 

Although the memory demand is small, performance is constrained by global memory bandwidth. Each thread performs repeated read–write operations to the shared distance arrays across all $N$ sources in every relaxation step. As $N$ increases, these concurrent accesses lead to high write contention and non-coalesced access patterns, saturating the GPU’s memory bus rather than its computational resources. This bottleneck is intrinsic to the Bellman--Ford algorithm when expressed in its fully parallel form: the very feature that enables massive edge-level concurrency also enforces global synchronization after every iteration. Any attempt to restructure the CUDA kernel to alleviate this effect would inherently reduce the degree of parallelism, negating the algorithm’s primary performance advantage. 
Execution times were measured using CUDA event timers. For the configuration above, the complete GPU execution---including $(V-1)$ relaxation rounds and host synchronization---completed in approximately \texttt{100}~ms, corresponding to a throughput of \texttt{2627}~million traversed edges per second (MTEPS). Profiling confirmed that the kernel is memory-bound, with arithmetic utilization remaining below 20\% of theoretical peak. 
All implementation details and datasets used for validation are available for reproducibility at \texttt{https://github.com/mnbe1973/Bellman-Ford-Cuda-C}.

\begin{rmk}
For validation, we verified that the CPU implementations of Dijkstra’s and Bellman--Ford algorithms produced identical shortest paths on representative test graphs before porting Bellman--Ford to CUDA~C. This confirmed functional equivalence between the two approaches. Dijkstra’s algorithm, however, is not straightforward to express in a fully parallel GPU form due to its priority-queue dependence, whereas Bellman--Ford can relax all edges concurrently. Hence, the latter was chosen for GPU acceleration. Metaheuristic approaches such as ant colony optimization (ACO) were deemed outside the scope of this study. Preliminary CPU--GPU comparisons indicated numerical consistency, while performance gains align with previously reported GPU speedups for edge-relaxation-based algorithms.
\end{rmk}

\section{Large-scale experiments}\label{sec:large_experiments}

To assess the scalability of the proposed method, we perform an additional set of experiments on a warehouse configuration of size $100 \times 100 \times 10$, corresponding to $100\,000$ nodes. The parameter settings otherwise follow those of Section \ref{sec:numericalexperiments}: each picking order consists of a group of 20 customer orders, each with 10 article types and 1--10 parcels of every such article type. This design preserves comparability with Experiments 1--3 above, while allowing us to examine performance in a substantially larger fulfillment center. In this section we thus have the following parameter settings:

\begin{enumerate}
\item \label{settings:11} The warehouse contains $100\times 100\times 10$ picking nodes located at $(i,j,k)\in [1,100]\times [1,100]\times[1,10]$. This means that $n_x=n_y=100$ and $n_z=10$ in \eqref{eq:dim}. 
\item \label{settings:12} The number of article types is $N=89\, 000$, and we assume that the maximum number of articles at a node is $10$ for all article types, so $M_{a_{ijk}}=10$ for all $(i,j,k)$. Thus, the state space $X$ is 
$$
X=\{(a_{ijk},m_{ijk})_{ijk}: a_{ijk}\in[0,89\, 000],\ m_{ijk}\in[0,10]\}.
$$
\item \label{settings:13} The initial warehouse state $x_0$ is randomized with the following constraints: It contains $1100$ $xy$-coordinate pairs $(i_l,j_l)$, $0\le l\le 1099$, such that the nodes $(i_0,j_0,k),\ldots,(i_{1099},j_{1099},k)$ are empty for all $k\in [1,10]$. This corresponds to $1100$ empty racks, for a total of $11\,000$ empty nodes. A unique article type $a_{ijk}\in[1,89\,000]$ is assigned to each of the remaining $89\,000$ nodes, and we assume that the initial warehouse state is fully stocked, so the balance at each nonempty node $10$. 
\item \label{settings:14} An order $o$ to be picked is a collection of $20$ individual customer or purchase orders. A purchase order consists of $10$ article types with between $1$ and $10$ parcels of each type. The $20$ purchase orders are mutually disjoint, so $o$ consists of $200$ article types. In each picking iteration, the $20$ purchase orders are divided into three groups using $K$-means clustering based on where the articles are stored in the warehouse at that time (with $K=3$ reflecting the classical ABC demand classification scheme, as discussed in Section~\ref{ss:relatedworks}). Each such group is a cluster. 
\end{enumerate}

An important consequence of this scaling is that the ratio of picked article types to the total number of available article types is now much smaller than in the previous $10 \times 10 \times 10$ setup. This provides a test of how clustering performs when orders occupy only a small fraction of the storage capacity. Under these conditions, we observed that the method performed well across all noise levels, including Experiment~3, which previously showed only marginal improvement in the smaller warehouse setting.

Due to the larger problem size, we restrict attention in this section to quantitative evaluation using silhouette scores and corresponding statistical validation. In contrast to Section~4, we do not include visual illustrations of cluster compactness or separation, but instead report results directly through statistical measures of cluster quality. Justification for the choice of $100 \times 100 \times 10$ nodes is provided below.

Figure~\ref{fig:silhouette_large} and Table~\ref{table2} summarize the results for the large warehouse configuration. Substantial improvements in silhouette score were observed in all experiments, with the largest gains in Exp~1 ($\Delta = 0.86$, 95\% CI [0.81, 0.90]), followed by Exp~2 ($\Delta = 0.50$, 95\% CI [0.45, 0.54]) and Exp~3 ($\Delta = 0.18$, 95\% CI [0.15, 0.21]). Pairwise permutation tests confirmed that all differences between experiments were highly significant ($p < 0.0001$ after Bonferroni correction). Effect sizes were uniformly large (Cohen’s $d > 5.5$, Cliff’s $\delta = 1.0$), indicating clear separation between noise levels. These results show that the clustering method scales effectively to a warehouse with $100\,000$ nodes and maintains stable performance even under the highest tested noise setting.

\begin{figure} 
\begin{center}
\begin{tikzpicture}
\node[inner sep=0pt, xshift=-6mm] (img) {\includegraphics[width=0.61\linewidth]{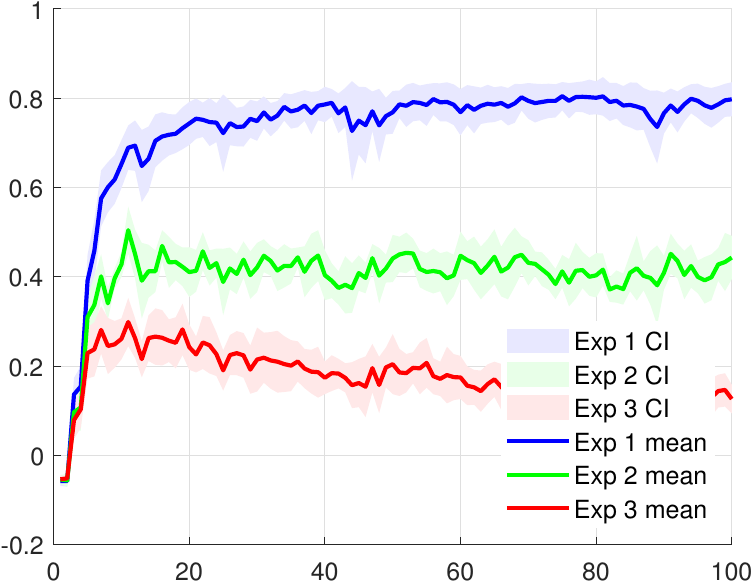}};
\node[rotate=90] at ([xshift=-4mm,yshift=1mm]img.west) {Silhouette score};
\node at ([xshift=2.2mm,yshift=-3mm]img.south) {Iteration};
\node at ([xshift=1.5mm,yshift=2mm]img.north) {\bf Silhouette trajectories with 95\% CI};

\path ([xshift=0mm]img.west) -- ([xshift=6.5mm]img.east) coordinate[pos=0](bbwest) coordinate[pos=1](bbeast);
\useasboundingbox (bbwest |- img.south) rectangle (bbeast |- img.north);
\end{tikzpicture}
\end{center}
\caption{Silhouette score trajectories over 100 iterations for the large warehouse configuration ($100 \times 100 \times 10$ nodes). Solid lines show the mean silhouette score across runs, with shaded bands indicating 95\% confidence intervals. All experiments show substantial improvements, with higher noise (Exp~3) leading to smaller gains compared to Exp~1 and Exp~2. \label{fig:silhouette_large}}
\end{figure}

\begin{table} 
\centering
\caption{Silhouette score improvements ($\Delta$) with 95\% confidence intervals, and pairwise statistical comparisons between experiments for a large warehouse with $100\, 000$ nodes. Pairwise $p$-values are from permutation tests (Bonferroni corrected). Effect sizes are reported as Cohen's $d$ and Cliff's $\delta$. \label{table2}}
\begin{tabular}{lccc}
\toprule
 & Exp 1 & Exp 2 & Exp 3 \\
\midrule
Mean improvement $\Delta$ & 0.86 & 0.50 & 0.18 \\
95\% CI & [0.81, 0.90] & [0.45, 0.54] & [0.15, 0.21] \\
\midrule
\multicolumn{4}{l}{\textbf{Pairwise comparisons}} \\
\midrule
Exp 1 vs Exp 2 & \multicolumn{3}{l}{$p < 0.0001$, $d = 5.6$, $\delta = 1.0$} \\
Exp 1 vs Exp 3 & \multicolumn{3}{l}{$p < 0.0001$, $d = 12.2$, $\delta = 1.0$} \\
Exp 2 vs Exp 3 & \multicolumn{3}{l}{$p < 0.0001$, $d = 5.6$, $\delta = 1.0$} \\
\bottomrule
\end{tabular}
\end{table}

In academic benchmarking, Duque-Jaramillo et al.~\cite{duque2024warehouse} document a distribution warehouse with two sections, 24 rows per section, 66 columns per row and seven levels per column, giving a total of 22\,176 storage slots and overall dimensions of $72 \,\text{m} \times 190 \,\text{m}$. This is approximately 22\% of the $100\,000$ nodes scale in our large instance and about 22 times the 1000 node scale in our small instance. Alqahtani~\cite{alqahtani2023improving} analyses a retail distribution centre handling 40 SKUs (32 in detailed analysis) within a mixed block-storage and selective-rack layout, equivalent to roughly 0.04\% of our large instance and 4\% of our small instance.

In industry benchmarking, Digi-Key’s Product Distribution Center expansion (PDCe) in Thief River Falls, Minnesota, is a 2.2 million ft$^2$ ($\approx 204\,387 \,\text{m}^2$) facility—one of the ten largest warehouses in North America—designed to enable nearly triple its previous daily average of 27\,000 packages, with dual conveyor systems spanning more than 27 miles~\cite{digikey2022,digikey2023}. While facility-level SKU counts are not public, its throughput capacity of approximately 81\,000 packages/day allows for a conservative order-of-magnitude estimate. Assuming an average of 5--7 SKUs per package, typical for electronics distributors, this corresponds to roughly $486\,000$ SKU-picks per day. Although SKU-picks/day is not identical to unique SKUs stored, this approximate throughput represents about five times the SKU-level scale of our large-instance simulation.

These comparisons place our large-instance simulation within the lower order of magnitude of SKU or picking scales observed in major operational warehouses, while also highlighting heterogeneity in SKU density, throughput, and layout. Detailed SKU-level and throughput data for the largest operators, including Amazon, are not publicly disclosed at the facility level due to commercial confidentiality, a limitation also noted by Zhuang et al.~\cite{zhuang2024improving}. Taken together, these benchmarks indicate that the scale of our simulated instance is representative of modern fulfillment operations while remaining computationally tractable for GPU-based optimization.

\section{Discussion and conclusions}\label{sec:discussion}

The results presented above demonstrate that the proposed clustering--RDS framework can reorganize warehouse layouts dynamically, yielding substantial improvements in route efficiency and cluster compactness. While the experiments rely on synthetic but representative warehouse configurations, the observed effects are systematic and robust across repetitions. Although the simulation setting abstracts from several operational complexities, it captures the essential structural dynamics of storage-location assignment and product reorganization. 

The following discussion outlines the robustness and limitations of the proposed framework, evaluates its practical relevance, and identifies directions for future research and industrial deployment.

The implementation of the Bellman--Ford algorithm in CUDA~C demonstrates that GPU-based parallelism is highly effective for solving routing problems, particularly when evaluating large numbers of path combinations. While the exponential growth of permutations presents a challenge, especially as the number of nodes increases, this limitation can be mitigated by dividing large routing problems into smaller sub-routes. By processing these segments independently, it is possible to maintain high computational efficiency without exceeding the GPU's memory or thread capacity. This segmentation strategy ensures that GPU acceleration remains a scalable and practical solution even for complex graph traversal problems. Therefore, the combination of Bellman--Ford and CUDA~C provides a flexible framework that balances algorithmic precision with hardware-aware execution.

The numerical experiments demonstrate that the proposed clustering--RDS framework remains robust under a wide range of order variability and warehouse dynamics. In smaller-scale experiments, performance declines as stochastic noise increases, yet clustering effects remain statistically significant and clearly measurable through both geometric and silhouette-based indicators. In the large-scale setting with $100\,000$ nodes, the method exhibits strong improvements in cluster compactness and separation across all noise levels, showing that its core mechanisms scale effectively even when orders occupy only a small fraction of total capacity. This robustness arises from the adaptive interaction between clustering and relocation: as orders fluctuate, the RDS formulation continuously rebalances the spatial distribution of products to maintain compact storage regions. Such behavior mirrors the adaptive, self-organizing tendencies of modern fulfilment systems, in which storage assignments evolve dynamically in response to changing demand.

While the framework captures key structural dynamics of warehouse reorganization, it relies on several simplifying assumptions that delimit its immediate operational realism. In particular, the model currently assumes single-SKU storage locations (one article type per node), and that articles may be relocated freely at each iteration. These abstractions allow the effects of clustering and route optimization to be studied in isolation, providing a clear baseline for understanding how spatial reorganization improves efficiency. In practice, relocations are constrained by available labour and may involve non-negligible delays. Such constraints can be incorporated in at least two ways: by placing explicit limits on the number of relocations per iteration (labour capacity constraints), or by assigning a relocation cost or delay to each move (time-based constraints). As discussed in Section~2.1, the mathematical formulation can be extended to incorporate multi-SKU nodes, which together with extensions such as relocation costs and time-based capacity limits represent a natural next step toward operational deployment of the model in real fulfilment environments.

In addition, replenishment---the restocking of depleted picking locations from reserve storage---is implicitly represented in our model through the relocation mechanism itself: when inventory of a given article type at a node is exhausted, the model replenishes it dynamically by assigning a new storage location based on the current clustering configuration. This can be interpreted as a form of adaptive, demand-driven replenishment rather than a static restocking process. Alternatively, replenishment operations could be modeled explicitly as delayed or capacity-limited relocation events, which would constitute natural extensions of the present model and could be integrated with our clustering framework in future work.

Another direction of future research concerns energy-aware warehouse operations, where sustainable storage assignment and system design are increasingly studied in the context of automated systems~\cite{rizqi2024energy}; such considerations could in principle be combined with clustering-based reorganization to jointly optimize efficiency and sustainability.

Comparable principles are already applied in industry.
For instance, Amazon’s robotic picking algorithms dynamically re-cluster storage units based on changing order patterns  \cite{allgor2023algorithm}, while Alibaba’s GPU-accelerated routing systems achieve real-time optimization at large scale \cite{hu2022alibaba}.
These examples demonstrate that the computational and organizational mechanisms analyzed here align with successful industrial practices, supporting the practical viability of the proposed framework.

Future research should also extend the framework toward operational deployment by embedding labour and equipment constraints directly into the RDS model through capacity or cost terms, enabling realistic scheduling of human or robotic pickers. Further, coupling the model with digital-twin or WMS simulation platforms would facilitate real-time decision support, enabling policy evaluation and benchmarking. Empirical validation on real warehouse datasets---beyond the synthetic but representative instances studied here---would establish quantitative benchmarks for cost savings and throughput gains, paving the way for practical adoption in large-scale industrial settings.

\subsection*{Managerial implications and insights}

From a practical standpoint, these findings translate into several actionable insights for warehouse managers. The proposed framework serves as a strategic decision-support tool for warehouse planners, logistics engineers, and managers responsible for optimizing storage layouts and order-picking efficiency. By integrating dynamic clustering with route optimization, it provides quantitative insight into how reorganizing stock locations can improve picking performance as demand patterns evolve. The model can be embedded into a WMS platform or used in an offline planning environment to evaluate alternative zoning and replenishment strategies before physical implementation, thereby identifying configurations that minimize travel distances and enhance space utilization.

The numerical experiments demonstrate that the clustering-based reorganization substantially improves order consolidation and reduces average route length, with the approximate GPU-accelerated routing procedure yielding paths within 2--14\% of the optimal solution while maintaining sub-second runtimes for problem sizes up to $10^5$ nodes. These results indicate that dynamic, data-driven reallocation of stock locations---implemented even periodically rather than continuously---can produce meaningful operational gains without major infrastructure changes. Compared to traditional static ABC zoning, the proposed method adapts automatically to shifting order profiles, maintaining compact and well-separated product clusters even under moderate noise conditions.

From a managerial perspective, the framework’s simplicity and modularity make it well suited for exploratory “what-if” analyses in digital-twin environments or as an analytical layer within existing WMS platforms. Implementation requires only basic positional data and historical order information, making it both accessible and interpretable for practitioners. In practice, further calibration would be needed to incorporate labour availability, relocation costs, and replenishment schedules---elements that can be introduced as constraints or penalty terms in future extensions. Overall, the framework provides a scalable, computationally efficient foundation for developing adaptive and energy-aware storage assignment policies in next-generation smart warehouses. Taken together, these results demonstrate how mathematically grounded models can directly inform data-driven warehouse management, bridging analytical rigor and managerial practice.

\appendix

\section{WMS loop}\label{sec:WMS}
Here we describe the algorithms of the WMS.
The graph $G$ has all information about the position of each product, its current inventory balance, as well as empty positions that can be used for re-positioning a product so that it can be situated closer to a cluster center. (In the mathematical language of \S\ref{sec:model}, the graph is thus a state $x\in X$.) The order graph consists of similar setup as the graph $G$, with the difference that this graph only has the order content (each article number) and the ordered amount of each product. (The order graph is thus an order $o$.) The WMS loop has one purpose, to enable orders that belong to a cluster to come closer to each other. This will in an ideal operating condition result in areas in the warehouse where only orders belonging to a specific cluster is situated, see Figure \ref{fig1and2}.

For clarity, Table~\ref{tab:notation_algorithms} summarizes the notation used in Algorithms~\ref{alg:process_order}--\ref{alg:main}.

\begin{table}[htb!]
\centering
\caption{Notation used in Algorithms~\ref{alg:process_order}--\ref{alg:main}.}
\label{tab:notation_algorithms}
\begin{tabular}{ll}
\toprule
\textbf{Symbol / Function} & \textbf{Description} \\
\midrule
$G$ & Warehouse graph containing product positions and balances \\
$order\_id$ & Identifier of the current order being processed \\
$ordergraf$ & Order graph containing article numbers and ordered quantities \\
$order\_index$ & Total number of orders in the current batch \\
$n\_repetitioner$ & Number of repetitions (iterations) of the main loop \\
$order\_string$ & Encoded string representation of an order \\
$test\_pos$ & Position index used for testing or stock updates \\
$next\_order\_string$ & Encoded representation of the subsequent order \\
$empty\_nodes$ & Set of empty warehouse nodes available for relocation \\
$blocked\_orders$ & Set of orders blocked due to insufficient stock \\
$specific\_point$ & Cluster center associated with the current order \\
$distance$ & Distance list between empty nodes and cluster center \\
$distance2$ & Distance list between blocked orders and cluster center \\
$min\_distance$ & Minimum distance among candidate relocation nodes \\
\midrule
\texttt{ProcessOrder()} & Converts an order ID into a string and updates stock levels \\
\texttt{UpdateStock()} & Updates inventory in $G$ based on the processed order \\
\texttt{CheckStock()} & Returns \textbf{True} if sufficient stock exists for an order \\
\texttt{FindEmptyNodes()} & Identifies empty nodes available for article relocation \\
\texttt{CalculateDistance()} & Computes distances between node sets and a reference point \\
\texttt{MoveArticle()} & Moves an article from one node to another \\
\texttt{CheckStockAndMove()} & Combines stock checking and relocation to balance inventory \\
\texttt{Main()} & Main WMS loop coordinating order processing and relocation \\
\bottomrule
\end{tabular}
\end{table}

The \textit{Order Processing} algorithm converts an order ID into a string representation and determines the position to test. It then updates the stock based on the order string and increments the pick count in the graph $G$, see Algorithm \ref{alg:process_order}.

\begin{algorithm} 
\SetKwInput{KwInput}{Input}
\SetKwInput{KwOutput}{Output}
\DontPrintSemicolon

\KwInput{order\_id, ordergraf}
\KwOutput{order\_string, test\_pos}

\SetKwFunction{FProcessOrder}{ProcessOrder}
\SetKwFunction{FUpdateStock}{UpdateStock}

\SetKwProg{Fn}{Function}{:}{}
\Fn{\FProcessOrder{$order\_id, ordergraf$}}{
    $order\_string, test\_pos \leftarrow order\_to\_string(order\_id, ordergraf)$\;
    \FUpdateStock{$G, order\_string$}\;
    $increment\_pick\_count(G, order\_string)$\;
}
\caption{Order Processing \label{alg:process_order}}
\end{algorithm}

The \textit{Update Stock} algorithm updates the stock in the graph $G$ based on the provided order string. It ensures that the inventory levels are adjusted according to the orders processed, see Algorithm \ref{alg:update_stock}.

\begin{algorithm} 
\SetKwInput{KwInput}{Input}
\SetKwInput{KwOutput}{Output}
\DontPrintSemicolon

\KwInput{G, order\_string}
\KwOutput{None}

\SetKwFunction{FUpdateStock}{UpdateStock}

\SetKwProg{Fn}{Function}{:}{}
\Fn{\FUpdateStock{$G, order\_string$}}{
    $update\_stock\_from\_order(G, order\_string)$\;
}
\caption{Update Stock \label{alg:update_stock}}
\end{algorithm}

The \textit{Check Stock} algorithm verifies if there is sufficient stock in the graph $G$ to fulfill the next order string. It returns a Boolean value indicating whether the stock is sufficient, see Algorithm \ref{alg:check_stock}.

\begin{algorithm}
\SetKwInput{KwInput}{Input}
\SetKwInput{KwOutput}{Output}
\DontPrintSemicolon

\KwInput{G, next\_order\_string}
\KwOutput{Boolean}

\SetKwFunction{FCheckStock}{CheckStock}

\SetKwProg{Fn}{Function}{:}{}
\Fn{\FCheckStock{$G, next\_order\_string$}}{
    \KwRet $is\_sufficient\_stock\_for\_order(G, next\_order\_string)$\;
}
\caption{Check Stock \label{alg:check_stock}}
\end{algorithm}

The \textit{Check Stock and Move Articles} algorithm first checks if there is sufficient stock for the next order. If the stock is insufficient, it identifies empty nodes and calculates the distances to determine the optimal movement of articles to balance the stock, see Algorithm \ref{alg:check_and_move}.

\begin{algorithm} 
\SetKwInput{KwInput}{Input}
\SetKwInput{KwOutput}{Output}
\DontPrintSemicolon

\KwInput{G, ordergraf, order\_id}
\KwOutput{None}

\SetKwFunction{FCheckStock}{CheckStock}
\SetKwFunction{FCheckStockAndMove}{CheckStockAndMove}
\SetKwFunction{FFindEmptyNodes}{FindEmptyNodes}
\SetKwFunction{FCalculateDistance}{CalculateDistance}
\SetKwFunction{FMoveArticle}{MoveArticle}

\SetKwProg{Fn}{Function}{:}{}
\Fn{\FCheckStockAndMove{$G, ordergraf, order\_id$}}{
    $next\_order\_string \leftarrow order\_to\_string(order\_id + 1, ordergraf)$\;
                
    \If{\FCheckStock{$G, next\_order\_string$} = \textbf{False}}{
        $empty\_nodes \leftarrow \FFindEmptyNodes{$G$}$\;
        $specific\_point \leftarrow ordergraf.nodes[order\_id].get('cluster\_center', 'Unknown')$\;
        $blocked\_orders \leftarrow update\_stock\_from\_order2(G, next\_order\_string)$\;
        $distance \leftarrow \FCalculateDistance{$empty\_nodes, specific\_point$}$\;
        $min\_distance \leftarrow find\_min(distance)$\;
        $distance2 \leftarrow \FCalculateDistance{$blocked\_orders, specific\_point$}$\;
                    
        \For{$i \leftarrow 0$ \KwTo $length(distance2) - 1$}{
            $test\_pos \leftarrow distance2[i][1] > min\_distance[1]$\;
                        
            \If{$test\_pos$}{
                \FMoveArticle{$G, distance2[i][0], min\_distance[0], blocked\_orders[i][3]$}\;
            }
        }
    }
}
\caption{Check stock and Move Articles \label{alg:check_and_move}}
\end{algorithm}

The \textit{Main Order Processing Algorithm} coordinates the overall process of handling orders. It iterates through a specified number of repetitions and processes each order by converting it to a string, updating the stock, and checking the stock to move articles as necessary, see Algorithm \ref{alg:main}.

\begin{algorithm} 
\SetKwInput{KwInput}{Input}
\SetKwInput{KwOutput}{Output}
\DontPrintSemicolon

\KwInput{ordergraf, G, order\_index, n\_repetitioner}
\KwOutput{None}

\SetKwFunction{FMain}{Main}
\SetKwFunction{FProcessOrder}{ProcessOrder}
\SetKwFunction{FCheckStockAndMove}{CheckStockAndMove}

\SetKwProg{Fn}{Function}{:}{}
\Fn{\FMain}{
    \For{$repetition \leftarrow 0$ \KwTo $n\_repetitioner - 1$}{
        \For{$order\_id \leftarrow 1$ \KwTo $order\_index$}{
            \FProcessOrder{$order\_id, ordergraf$}\;
            
            \If{$order\_id < number\_of\_orders$}{
                \FCheckStockAndMove{$G, ordergraf, order\_id, order\_string$}\;
            }
        }
    }
}
\caption{Main Order Processing Algorithm \label{alg:main}}
\end{algorithm}

\bibliographystyle{amsplain}
\bibliography{FoU_INDEK}

\end{document}